\documentclass{amsart}

\RequirePackage{float}
\RequirePackage[margin=1in,dvips]{geometry}
\RequirePackage{graphicx}
\RequirePackage{amsmath, amsthm, amssymb, amsfonts,slashed}
\RequirePackage[colorlinks=true]{hyperref}
\hypersetup{urlcolor=blue, citecolor=blue}
\RequirePackage[nameinlink]{cleveref}
\RequirePackage{xifthen}
\RequirePackage{subcaption}
\RequirePackage{indentfirst}
\RequirePackage{microtype}
\RequirePackage{xcolor}
\definecolor{mg}{RGB}{34, 139, 34}

\newenvironment{Eq}[1][]{\begin{equation}\ifthenelse{\equal{#1}{}}{}{\tag{#1}}\begin{aligned}}{\end{aligned}\end{equation}\ignorespacesafterend}
\newenvironment{Eq*}[1][]{\begin{equation*}\ifthenelse{\equal{#1}{}}{}{\tag{#1}}\begin{aligned}}{\end{aligned}\end{equation*}\ignorespacesafterend}
\numberwithin{equation}{section}
\newtheorem{theorem}{Theorem}[section]
\newtheorem{hypothesis}{Hypothesis}
\newtheorem{proposition}[theorem]{Proposition}

\newtheorem{lemma}[theorem]{Lemma}
\theoremstyle{remark}%
\newtheorem{remark}{Remark}[section] 
\theoremstyle{definition}

\newcommand{\al}{\alpha}    
    
\newcommand{\vep}{\varepsilon}

\newcommand{\R}{\mathbb{R}}
\newcommand{\C}{\mathbb{C}}

\newcommand{\M}{\mathcal{M}}
\newcommand{\Hg}{\mathcal{H}}

\def\<{\langle}             \def\>{\rangle}
\newcommand{\pa}{\partial}
\newcommand{\les}{{\lesssim}}
\newcommand{\beeq}{\begin{equation}}\newcommand{\eneq}{\end{equation}}
\newcommand{\varE}{{\mathfrak E}}
\newcommand{\varF}{\varE'}

\DeclareMathOperator{\fd}{d}	\renewcommand{\d}{\fd}

\newcommand{\Roma}[1]{\uppercase\expandafter{\romannumeral#1}}

\newcommand{\sGamma}{{\slashed \Gamma}}

\newcommand{\sg}{{\slashed g}}
\newcommand{\snabla}{{\slashed \nabla}}

\newcommand{\tw}{{\tilde{w}}}
\newcommand{\tpartial}{{\tilde{\partial}}}

\newcommand{\rR}{{\mathbb{R}}}
\newcommand{\rH}{{\mathbb H}}
\newcommand{\rS}{{\mathbb{S}}}

\newcommand{\spartial}{{\slashed \partial}}
\newcommand{\hpartial}{{\hat \partial}}

\newcommand{\kl}[1]{\mathopen{}\left#1}
\newcommand{\kr}[1]{\right#1}

\crefname{equation}{}{}
\newcommand{\Tm}[1]{Theorem \ref{#1}}

\newcommand{\La}[1]{Lemma \ref{#1}}

\newcommand{\Sn}[1]{Section \ref{#1}}

\newcommand{\Hs}[1]{Hypothesis \ref{#1}}

\title[Wave equations on hyperbolic space]
{Global existence for quasilinear wave equations on hyperbolic space}

\author{Wei Dai}
\address{
School of Mathematical Sciences\\
 Zhejiang University of Technology\\ 
 Hangzhou 310023,
 China}
\email{daiw23@zjut.edu.cn}

\author{Chengbo Wang}
\address{School of Mathematical Sciences\\
Zhejiang Normal University\\
Jinhua 321004, China}
\email{wangcbo@zjnu.edu.cn}

\date{\today}

\begin{document}
\bibliographystyle{plain}

\begin{abstract}
The main purpose of this paper is to study the global solvability for a general class of quasilinear \textit{shifted} wave equations on hyperbolic spaces, for  smooth initial data with small amplitude. 
In contrast to the case of Euclidean spaces, when the space dimension is three, we do not need to assume structural conditions like the null conditions to ensure global existence.
To achieve this, we establish the energy and local energy estimates for perturbed wave operators on $\rR\times \rH^n$. These estimates allow time-dependent metric perturbations and require only suitable smallness together with polynomial decay in the radial variable $r$. 
As a byproduct,  for semilinear problems with power-type nonlinearities and radial data, 
we also obtain global solutions with low-regularity. In particular, we prove an analog of the radial Glassey conjecture on  hyperbolic space.
\end{abstract}



\maketitle

\section{Introduction}\label{Sn:I}

The main purpose of this paper is to establish small-data global existence results for the Cauchy problem for a class of typical quasilinear \textit{shifted} wave equations on hyperbolic space:
\begin{Eq*}
\square_{\rH}\phi=F(\phi,\partial_t\phi,\partial_t^2\phi).
\end{Eq*}
For this purpose, we derive energy and local energy estimates adapted to perturbed wave operators.
It is worth emphasizing that these estimates do not require the perturbation to be static; rather, they only assume
polynomial decay in the geodesic distance
together with suitable smallness conditions.
Finally, as a byproduct, we also present related results on the critical exponents for semilinear problems with power-type nonlinearities, analogous to those appearing in the study of the \emph{Strauss} and \emph{Glassey} conjectures in Euclidean space.

\subsection{Hyperbolic space}
We begin by recalling the geometric setting and introducing the class of equations under consideration.
As is well known, the $n$-dimensional (real) hyperbolic space $\rH^n$ is a complete simply connected manifold of constant negative curvature $-1$, 
and it can be viewed as an embedded submanifold of Minkowski space. 
To be more specific, inside the forward light cone $\Lambda=\{(\tau,y)\in \rR\times\rR^n: |y|<\tau\}$, we introduce the new coordinates $(s,r,\omega)$ satisfying
\begin{Eq*}
\tau=e^s\cosh r,\qquad y=e^s\sinh r\omega,\qquad s\in\rR,\quad r\in[0,\infty),\quad \omega\in\rS^{n-1}.
\end{Eq*}
Then the hyperbolic space $\rH^n=\Lambda_{s=0}$ is the submanifold endowed with the metric induced from the Minkowski metric $m=-\d\tau^2+\d y^2$.

Let $\Theta=(\theta^A)_{A=1}^{n-1}$ be the standard orthogonal coordinate system on the $(n-1)$-dimensional sphere $\rS^{n-1}$. 
Then the standard metric on $\rH^n$ is given by
$\d r^2+\sinh^2 r \d\Theta^2$, 
where
\begin{Eq*}
\d\Theta^2:=\sg_{AB}\d\theta^A \d\theta^B
=(\d\theta^1)^2+\sum_{A=2}^{n-1}\kl(\prod_{D=1}^{A-1}\sin^2\theta^D\kr)(\d\theta^A)^2.
\end{Eq*}

Let $\M_0=\rR\times \rH^n$ be the spacetime manifold equipped with the metric
\begin{Eq}\label{Eq:g0}
g_0=-\d t^2+\d r^2+\sinh^2 r \d\Theta^2.
\end{Eq}
With $\rho=(n-1)/2$, the natural \textit{shifted} wave operator is defined by
\begin{Eq*}
\square_{\rH}:=e^{(2+\rho)t}\square_m e^{-\rho t}
=-\partial_t^2+\Delta_{\rH}+\rho^2
=\nabla^\alpha\nabla_\alpha+\rho^2.
\end{Eq*}
Here and in what follows, we adopt the Einstein summation convention. 
Capital indices $A,B,\dots$ range from $1$ to $n-1$, 
while Greek indices $\alpha,\beta,\dots$ run over $t$, $r$, 
and the angular variables. 
Indices are raised and lowered using the metric  $g_0$.
 Here $\nabla$ denotes the Levi-Civita connection of $g_0$.
Moreover,
\begin{Eq*}
\Delta_{\rH}=\partial_r^2+\frac{n-1}{\tanh r}\partial_r+\frac{1}{\sinh^2 r}\Delta_{\rS},
\qquad
\Delta_{\rS}=\sg^{AB}\snabla_A\snabla_B,
\end{Eq*}
where $\snabla$ denotes the covariant derivative associated with the metric $\sg$.

\subsection{Quasilinear wave equations}
We are now ready to present the main problem studied in this paper. 
More precisely, we consider quasilinear wave equations of the form
\begin{Eq}\label{Eq:P1}
\begin{cases}
\square_{\rH}\phi
=F(\phi,\partial_t\phi,\partial_t^2\phi)
=F_1(\phi,\partial_t\phi)\partial_t^2\phi+F_2(\phi,\partial_t\phi),\\
(\phi,\partial_t\phi)|_{t=0}
=\varepsilon(\psi_0,\psi_1).
\end{cases}
\end{Eq}
Here we assume that $F_1$ and $F_2$ are both globally smooth, and satisfy
$F_1(0,0)=F_2(0,0)=F_2'(0,0)=0$.

Before presenting our main result on the quasilinear wave equations, let us briefly recall the rich history.
There is an extensive literature on the corresponding problems in Euclidean space.
 Let us recall only the related results for $n\ge 3$. When the nonlinearity does not depend on the unknown $\phi$, it is known that we have global existence for $n\ge 4$ and almost global existence
($T_\varepsilon\ge \exp(c\varepsilon^{-1})$), 
 for $n=3$,
 where $\varepsilon$ is the amplitude of the initial data,
  see John--Klainerman \cite{JK84}, Klainerman \cite{MR784477} and references therein. 
When $n=3$, in general, we can only expect almost global existence, as John \cite{Jo81} proved blow-up at almost global time scales for $\Box \phi=\phi_t^2$ and $\Box \phi=\phi_t\phi_{tt}$.
 On the other hand, when the nonlinearity satisfies Klainerman's null condition, the problem admits global solutions, 
see Klainerman \cite{Klai86} and Christodoulou \cite{MR820070}.

When the nonlinearity depends also on the unknown $\phi$, the situation is even more complicated. Global existence is known for $n\ge 5$. For $n=4$, if the term $\phi^2$ is absent, we also have global existence. Otherwise, there is almost global existence, with lower bound of the lifespan $T_\varepsilon\ge \exp(c\varepsilon^{-2})$, see H\"ormander \cite{MR1120284}, Li--Zhou \cite{MR1386767}, Lindblad--Sogge \cite{LdSo96}. When $n=3$, we have $T_\varepsilon\ge c\varepsilon^{-2}$ in general and $T_\varepsilon\ge \exp(c\varepsilon^{-1})$ when  the term $\phi^2$ is absent, see Lindblad \cite{Ld90-2QLW}.
For a thorough discussion of the history, we refer the readers to Sogge \cite{So08} and
 Li--Zhou
\cite{MR3729480}.
For $n=3$, the general structure conditions which ensure global existence remain open. Despite that, we have the so-called weak null conditions, which are closely related to global existence and  have drawn much attention, see Lindblad--Rodnianski \cite{LdRo03}, Alinhac \cite{MR2244605}, Lindblad--Tohaneanu \cite{MR4810084} and references therein.

Turning to hyperbolic space, to the best of the authors'
knowledge, there are not many works concerning the nonlinear (shifted) wave equations. Most of them are devoted to the semilinear wave equations with power-type nonlinearity in the unknown:
$$
\square_{\rH}\phi =F_p(\phi)\ ,$$
where $F_p$ satisfies $|F_p(\phi)|+|\phi F'_p(\phi)|\le C|\phi|^p$ for some $p>1$.
Recall that, in Euclidean space $\R^3$, we have finite-time blow-up for $\Box \phi=\phi^2$, whenever the initial data is nontrivial, see John \cite{MR526180}.
In contrast to the Euclidean case, thanks to the improved dispersive estimates on hyperbolic space, we have global existence for
$$
\square_{\rH}\phi =\phi^2$$
for any spatial dimension $n\ge 2$, see Fontaine \cite{Fo97} for $n=2,3$ and
Sire--Sogge--Wang \cite{SSW19} for 
$n\ge 2$\footnote{Direct application of \cite[Theorem 1.4]{SSW19} gives us the result for $n\in [2, 7]$. For the smooth nonlinearity, $\phi^2$, in the high dimensional case, the argument could essentially be adapted to show global existence. For example, for $n\ge 3$, we could choose $q\in (2, 2(n-1)/(n-2))$, $p_0=2$ and $p_1=\infty$ in \cite[(36)-(37)]{SSW19}, which is strong enough to yield global existence, by taking derivatives of sufficient high order (say $3n/q-n$) to the equation.}.

Inspired by these results, it seems natural to expect global results for more general nonlinearities, compared with the Euclidean case. In particular, it is interesting to see if the null condition is still required for $n=2, 3$.


The following is our first main result, which suggests that we do not need to assume any structural conditions to ensure global existence for quasilinear wave equations on hyperbolic space, at least for $n\ge 3$.

\begin{theorem}\label{Tm:P1}
Let $n\geq 3$. Consider the Cauchy problem \eqref{Eq:P1}. 
Assume that the initial data $\psi_0,\psi_1\in C_0^\infty(\rH^n)$. 
Then, if $\varepsilon>0$ is sufficiently small, 
the Cauchy problem admits a global smooth solution.
\end{theorem}
\begin{remark}
The argument actually applies to a broader class of quasilinear equations. 
For instance, one may allow $F_1$ and $F_2$ to depend on the coordinates $(t,r,\Theta)$ or on spatial derivatives of $\phi$, 
and one may also consider perturbations involving purely spatial or more general spacetime second-order derivatives. 
However, since hyperbolic space has no Killing vector fields of uniformly bounded size analogous to the Euclidean translations $\partial_i$, 
treating the full generality would introduce additional technicalities. 
For this reason, we restrict ourselves here to the model case \eqref{Eq:P1}, 
which already captures the main idea.
\end{remark}

As in the Euclidean case, one naturally expects to prove global-in-time existence by combining energy and local energy estimates with an argument analogous to the \emph{Klainerman--Sobolev} inequality, see, e.g., Keel--Smith--Sogge \cite{KSS04},
Metcalfe--Sogge \cite{MetSo06},
Hidano--Wang--Yokoyama \cite{HWY1}, and
Wang \cite{MR4662297}.
It is well known that local energy estimates are of fundamental importance and have a wide range of applications; see, for instance,
Metcalfe--Tataru \cite{MeTa12MA},
Andersson--Blue \cite{MR3418531}, 
Dafermos--Rodnianski--Shlapentokh-Rothman \cite{MR3488738},
Metcalfe--Sterbenz--Tataru
\cite{MR4101333} and
Ma \cite{MR4062461}. 

Suitable resolvent estimates can lead to local energy estimates
in the setting of conformally compact
asymptotically hyperbolic manifolds (cAHM), in the sense of Mazzeo--Melrose \cite{MM}; see  Section \ref{sec:cahm} for the definition and \Sn{Sn:A} for a more detailed discussion of the proof.
Similarly, there have also been many results on dispersive estimates and Strichartz estimates in hyperbolic space and cAHM; see, for instance, Tataru \cite{Ta01-2}, 
Anker--Pierfelice--Vallarino \cite{APV12}, Sire--Sogge--Wang--Zhang \cite{MR4169670, MR4432952}.
 However, these resolvent and dispersive tools typically require
the metric perturbation to be static and to decay very rapidly, and therefore
they are not suitable for the quasilinear problems considered here.
For this reason, we need to re-establish energy and local energy estimates adapted to more general perturbations.
We now state this key analytic estimate separately, since it may also be useful in other problems and will be used again for the semilinear applications below.

\subsection{Energy and local energy estimates for perturbed wave operators}
Let us first describe the spacetime manifold $(\M,g)$. 
Similar to Bony--H\"afner \cite{BoHa} and Sogge--Wang \cite{SW10} in the case of asymptotically Euclidean manifolds, 
we work on the same ambient manifold $\M=\M_0=\rR\times\rH^n$, 
but endow it with a different Lorentzian metric $g$. 
More precisely, in the coordinates $x=(t,r,\Theta)$, we write
$g=g_{\alpha\beta}(x)\d x^\alpha \d x^\beta$.

Let $(g^{\alpha\beta})$ be the inverse matrix of $(g_{\alpha\beta})$. We assume that $(g^{\alpha\beta})$ is a small perturbation of $(g_0^{\alpha\beta})$, and write
$\gamma^{\alpha\beta}:=g^{\alpha\beta}-g_0^{\alpha\beta}$.
We further assume that $\gamma^{\alpha\beta}\to 0$ as $r\to\infty$, so that the metric is asymptotically hyperbolic near spatial infinity.

To be slightly more general, we consider the perturbed wave operator
\begin{Eq}\label{Eq:eoP}
P:=\square_{\rH}+\nabla_\alpha(\gamma^{\alpha\beta}\nabla_\beta)+G^\beta\nabla_\beta+V,
\end{Eq}
in dimension $n\geq 3$, whose principal part agrees with that of the d'Alembertian operator $\square_g$. 
For the perturbation terms, we impose the following hypothesis.
\begin{hypothesis}\label{Hs:hoP}
Let $0<\varepsilon'\ll1$ and
\begin{Eq*}
M^\alpha:=\sqrt{|g_0^{\alpha\alpha}|}=\begin{cases}
 \kl(\sinh r\prod_{D=1}^{A-1}\sin\theta^{D}\kr)^{-1}, & \alpha=A,\\
1,& \al=t,r.
\end{cases}
\end{Eq*}
We assume that the metric perturbation $\gamma^{\alpha\beta} = \gamma^{\beta\alpha}$ satisfies the conditions
$|\gamma^{\alpha\beta} |\leq \varepsilon' M^\alpha M^\beta$ and 
\begin{Eq*}
\sum_{j=-\infty}^{\infty}\sup_{\rR\times A_j}
r\kl(\frac{|\partial_{t,r}\gamma^{\alpha\beta}|}{M^\alpha M^\beta}
+\kl<r\kr>\kl|\kl(
\partial_t\gamma^{tr},\partial_r\gamma^{rr},\snabla_A \gamma^{r A}
\kr)\kr|+
 r|\gamma^{rr}|
\kr)\le \varepsilon',
\end{Eq*}
where $A_j:=\{(r,\Theta):r\in[2^{j-1},2^{j+1}]\}$. 
Additionally, we assume that the lower-order perturbations $G$ and $V$ satisfy
\begin{Eq*}
\sum_{j=-\infty}^{\infty}\sup_{\rR\times A_j}
r\kl(
 | G^t|+ 
 \frac{|G^A|}{M^A}+
\kl<r\kr>|G^r|+
 r| V|
\kr)\leq \varepsilon'.
\end{Eq*}
\end{hypothesis}
Now, let
$w:=\frac{\rho}{\tanh r}=\frac{n-1}{2\tanh r}$.
We define the modified derivative of $f$ by
\begin{Eq*}
|\tpartial f|^2:=|\partial_t f|^2+\kl|\partial_r f+w f\kr|^2+|\spartial f|^2
=|\partial_t f|^2+\kl|\partial_r f+w f\kr|^2+\frac{\sg^{AB}\partial_A f\partial_B f}{(\sinh r)^2},
\end{Eq*}
which satisfies $\|\tpartial f\|_{L^2(\rH^n)} \approx \|\partial_t f\|_{L^2(\rH^n)} + \kl\|\sqrt{-\Delta_{\rH} - \tfrac{(n-1)^2}{4}}f\kr\|_{L^2(\rH^n)}$ when $n\geq 3$; see \Sn{Sn:eefuo} for a detailed discussion.
For later use, we also introduce the notation 
\begin{Eq*}
\hpartial_r f:=\partial_rf+2wf=\sinh^{1-n}r\partial_r(\sinh^{n-1}r f),
\end{Eq*}
which is the dual operator of $-\pa_r$ with respect to the $L^2(\M_0)$.

Next, on the time-space domain $S_0^T:=[0,T)\times\rH^n$,
we introduce the energy and local energy norms of $f$ as
\begin{Eq*}
\|f\|_{E(S_0^T)}:=\kl\|\kl(\tpartial f, \frac{f}r\kr)\kr\|_{L^\infty L^2(S_0^T)},\ 
\|f\|_{LE(S_0^T)}:=&\sup_{j}2^{-j/2}\kl\|\kl(\tpartial f, \frac{f}r\kr)\kr\|_{L^2L^2\kl([0,T] \times A_j\kr)}.
\end{Eq*}
Furthermore, we define the dual norm of local energy as
\begin{Eq*}
\|F\|_{LE^*(S_0^T)}:=\sum_{j}2^{j/2}\|F\|_{L^2L^2([0,T]\times A_j)}.
\end{Eq*}

We are now ready to present our result concerning the energy and local energy estimates.
\begin{theorem}[Energy and local energy estimates]\label{Tm:T}
Let $n\geq 3$. 
For the wave operator $P$ defined in \eqref{Eq:eoP}, 
whose coefficients satisfy \Hs{Hs:hoP}, we have
\begin{Eq}\label{Eq:Pgei}
\|f\|_{[E\cap LE](S_0^T)}\lesssim \|\tpartial f(0)\|_{L^2(\rH^n)}+\|P f\|_{[L^1L^2+LE^*](S_0^T)}
\end{Eq}
for any $T\in(0,\infty]$ and any compactly supported smooth function $f$,  
where the implicit constant depends only on $\varepsilon'$ arising in \Hs{Hs:hoP}.
\end{theorem}
As usual,
the notation $a\lesssim b$ (or equivalently $b\gtrsim a$) 
means that there exists a constant $C>0$ such that $a\leq Cb$,
where $C$ may vary from line to line.
Similarly, $a\approx b$ indicates that both $a\lesssim b$ and $b\lesssim a$ hold.

The idea of the proof is based on the multiplier method.
We first consider the unperturbed equation, for which an energy inequality can be obtained without difficulty.
However, unlike the wave equation in Euclidean space, the corresponding energy density here is not strictly positive.
To address this, we introduce a correction and thereby obtain a positive energy density involving the modified derivative $\tpartial f$ (see \Sn{Sn:eefuo}).
For technical reasons, the case $n=2$ is already excluded at this stage.

Next, by taking a weighted derivative in the $r$-direction as a multiplier, we establish an appropriate local energy estimate.
It should be emphasized that, in contrast to the Euclidean case, the local energy for $f$ arises from modifying the form of the corrected derivative rather than from a direct computation (see \Sn{Sn:leefuo}).

Finally, we repeat the above procedure for the perturbed wave equation.
Under suitable assumptions on the perturbation, we show that the error terms produced by the perturbation can ultimately be absorbed (see \Sn{Sn:efpo}).
This finally leads to \Tm{Tm:T}.

\begin{remark} 
Since the standard hyperbolic metric takes the form \eqref{Eq:g0}, one might naturally expect, by analogy with the Euclidean case, that energy or local energy should be weighted by $\sinh r$ and decomposed dyadically with respect to it. However, this theorem shows that the correct weight and dyadic decomposition are actually given by $r$ itself. This result directly explains why only polynomial decay of the metric perturbation is required in the present work.
\end{remark}

\subsubsection{Comparison with cAHM}\label{sec:cahm}
For comparison, we briefly recall the notion of conformally compact asymptotically hyperbolic manifolds (cAHM). Let $(\Hg^\circ,g)$ be a Riemannian manifold, where $\Hg$ is a compact manifold with boundary $\partial \Hg$ and $\Hg^\circ$ denotes its interior. Let $\lambda$ be a boundary defining function for $\Hg$. We say that $g$ is conformally compact if $\lambda^2 g$ extends smoothly to $\partial \Hg$ as a Riemannian metric. If, in addition,
$\lim_{\lambda\to 0+}|d\lambda|_{\lambda^2 g}^2=1$,
then $(\Hg^\circ,g)$ is said to be asymptotically hyperbolic.

More specifically, let $y=(y_1,\cdots,y_{n-1})$ be local coordinates on $\partial \Hg$. Then, in a collar neighborhood $(0,\epsilon)_\lambda\times \partial \Hg$, the metric $g$ takes the form
\begin{equation}\label{metric}
g=\frac{d\lambda^2}{\lambda^2}+\frac{h_{jk}(\lambda,y)dy^jdy^k}{\lambda^2},
\end{equation}
for some smooth family of metric tensors $h_{jk}$ on $\partial \Hg$.
 
A canonical example of this setting is hyperbolic space itself, for which
\begin{Eq*}
g=\frac{\d \lambda^2+\frac{1}{4}(1-\lambda^2)^2\d\Theta^2}{\lambda^2}\quad \text{on} \quad (0,1)\times\rS^{n-1},
\end{Eq*}
with $\d\Theta^2$ as above. Under the change of variables $\lambda=e^{-r}$, the metric can be written as
$\d r^2+\sinh^2 r\,\d\Theta^2$,
which is exactly the standard hyperbolic metric introduced earlier. By making a further change of variables such that
$
|z|=\frac{\sinh r}{1+\cosh r}$,
one also recovers the classical Poincar\'e disc model
$g=\frac{4\d z^2}{(1-|z|^2)^2}$ in $\{|z|<1\}$.

 It is well known that the sectional curvature of cAHM satisfies
$K=-1+O(\lambda)$.
Since $\lambda\sim e^{-r}$ near infinity, this indicates that, in terms of the radial variable $r$, a cAHM may be viewed as an exponential perturbation of the hyperbolic metric near spatial infinity.

Therefore, despite the smallness assumption, our framework is in several respects more general than the standard cAHM setting. First, we allow time-dependent perturbations. Second, even in the time-independent case, we do not require the metric to be static. Finally, even in the static case, our assumptions on the metric near spatial infinity require only polynomial decay and are therefore weaker than those in the cAHM framework.

\subsection{Semilinear wave equations with power-type nonlinearity}
In \Tm{Tm:P1}, we showed that quasilinear wave equations with standard smooth nonlinearities always guarantee the global existence of small-data solutions.
In Euclidean space, however, there is another classical class of semilinear wave equations whose nonlinearities are of power-type but have limited regularity:
\begin{Eq*}
\square_{\rR}\phi=F_p(\phi)~or~F_p(\partial\phi).
\end{Eq*}
Here the nonlinearity satisfies $|F_p(s)|+|s F_p'(s)|\les |s|^p$, and typical examples include $\pm|s|^{p-1}s$ and $\pm|s|^p$.

In particular, for nonlinearities of the form $|\phi|^p$, it is known that there exists a critical exponent
$p_S=\frac{n+1+\sqrt{n^2+10n-7}}{2(n-1)}$,
and hence global small-data solutions exist only when $p>p_S$ (and $p$ is not too large).
Such problems are usually referred to as the Strauss conjecture and related topics.
A more comprehensive discussion can be found in Wang \cite{Wang18}, including some corresponding results in asymptotically Euclidean spaces.

On the other hand, for nonlinearities of the form $|\partial\phi|^p$, it is known that there exists another critical exponent
$p_G=\frac{n+1}{n-1}$,
and hence global small-data solutions exist only when $p>p_G$.
Such problems are usually referred to as the Glassey conjecture and related topics.
A more comprehensive discussion can also be found in
Wang \cite{W15}, including some corresponding results in asymptotically Euclidean spaces.
It is worth noting that, for higher dimensions $n\geq 4$ and $p_G<p<2$, the regularity of this nonlinearity is too limited.
As a result, 
small-data global existence is only available in the radially symmetric setting.

More generally, results concerning combinations of these two types of nonlinearities can also be found in Hidano--Wang--Yokoyama \cite{HWY3}.
For the corresponding systems of equations, see for example Dai--Fang--Wang \cite{DFW19}.

We now return to hyperbolic space. In contrast with the Euclidean case, the geometry of $\rH^n$ provides stronger decay and integrability at spatial infinity. As a consequence, we expect that the critical-threshold picture familiar from the Strauss and Glassey conjectures changes substantially in the present setting. 

Thanks to the improved dispersive estimates \cite{Ta01-2, APV12}, the analogs of the Strauss conjecture on hyperbolic space have been fully investigated in
 \cite{Fo97, SSW19}, 
 where global existence was obtained for the sub-conformal range $1<p<1+4/(n-1)$.
Furthermore, under some natural spectral conditions, the improved Strichartz estimates and global results were extended for nontrapping cAHM, in \cite{MR4169670, MR4432952}.
 In addition, it turns out that the critical nature appears for the logarithmic nonlinearities, at least for $\rH^n$ with $n=2,3$, see
Wang--Zhang \cite{WZ25}
and Zhang
\cite{Z23-p}.

Somewhat surprisingly, there seem to be no corresponding advances for nonlinearities involving $\pa \phi$, which  is the part relevant to the Glassey conjecture. With the help of  the local energy estimates in \Tm{Tm:T}, we establish a radial small-data global existence result for such nonlinearities.
More precisely, we have the following theorem.
\begin{theorem}\label{Tm:P2}
Let $n\geq 3$, $1<p_1<\frac{n+2}{n-2}$, and $1<p_2,p_3<\frac{n}{n-2}$. Consider the equation
\begin{Eq}\label{Eq:P2}
\begin{cases}
P\phi=F_{p_1}(\phi)+F_{p_2}(\partial_t\phi)+F_{p_3}(\partial_r\phi),\\
\phi(0,r)=\varepsilon\psi_0(r),\quad \partial_t\phi(0,r)=\varepsilon\psi_1(r),
\end{cases}
\end{Eq}
where $P$ is the radially invariant operator defined in \eqref{Eq:eoP}.
Assume that
\begin{Eq}\label{Eq:hoP2}
\sup_{\rR\times\{r<1\}}\kl|\kl(\nabla^{\leq 2} \gamma,\nabla^{\leq 1}(G, V)\kr)\kr|
&\lesssim\varepsilon',\\
\sum_{j=0}^{\infty}\sup_{\rR\times A_j}r^2\kl|\kl(\partial_{t,r}^{\leq 2}\gamma,\partial_{t,r}^{\leq 1}(G,V)\kr)\kr|
&\lesssim \varepsilon'.
\end{Eq}
In the radial setting, these assumptions 
in particular imply \Hs{Hs:hoP}.
Assume that the radial initial data satisfy
\begin{Eq}\label{Eq:coid2}
\|\partial_r^{\leq 2}\psi_0\|_{L^2(\rH^n)}+\|\partial_r^{\leq 1}\psi_1\|_{L^2(\rH^n)}<1.
\end{Eq}
Then, for sufficiently small $\varepsilon$,
the equation admits a global solution $\phi$ satisfying
\begin{Eq*}
\|r^{-\frac{1}{2}}(\sinh r)^{\frac{n-1}{2}}\partial_{t,r}^{\leq 1}\phi\|_{L^\infty L^\infty(S_0^\infty)}+\|\partial_{t,r}^{\leq 1}\phi\|_{[E\cap LE](S_0^\infty)}\lesssim\varepsilon.
\end{Eq*}
\end{theorem}
Here $\nabla^{\leq m}f:=\{\nabla^{\alpha}f\}_{|\alpha|\leq m}$,
and analogous notation will be used for other derivatives throughout the paper.

\begin{remark}
The assumptions in \Tm{Tm:P2} can be relaxed in several directions. 
For instance, one may allow certain singular behaviors of $P$ near the origin, 
or weaken some of the decay assumptions on the lower-order terms at spatial infinity. 
We impose the present hypotheses only to keep the argument transparent and to avoid additional technical complications.
\end{remark}

The rest of the paper is organized as follows.
In Section~\ref{Sn:p}, we collect some preliminaries on hyperbolic geometry, vector fields, and functional inequalities that will be used throughout the paper.
In Section~\ref{Sn:qlw}, we prove \Tm{Tm:P1} by means of an iteration argument.
In Section~\ref{Sn:energyLE}, we derive the basic energy and local energy estimates for the unperturbed shifted wave operator.
In \Sn{Sn:efpo}, we establish the corresponding estimates for perturbed wave operators and prove \Tm{Tm:T}.
In Section~\ref{Sn:SLWgwp}, we apply these estimates to prove \Tm{Tm:P2} for 
 semilinear problems with power-type nonlinearities in the radial setting.
Finally, the appendix explains how resolvent estimates on cAHM imply a Kato-type weighted smoothing estimate, which may be viewed as a spectral analogue of local energy estimates.

\section{Preliminaries}\label{Sn:p}
 We first record the coordinate formulas for this covariant derivative:
\begin{Eq*}
\nabla_\alpha f=\partial_\alpha f,\qquad \nabla_\alpha X^\beta=\partial_\alpha X^\beta+\Gamma^\beta{}_{\alpha \sigma}X^\sigma,\qquad \nabla_\alpha\omega_\beta=\partial_\alpha \omega_\beta-\Gamma^\sigma{}_{\alpha\beta}\omega_\sigma,
\end{Eq*}
for a function $f$, a vector field $X$, and a one-form $\omega$, respectively.
Here $\Gamma$ denotes the Christoffel symbols of the metric $g_0$ in the coordinate system $(t,r,\Theta)$, whose nonzero components are given by
\begin{Eq*}
\Gamma^{A}{}_{r D}=&\frac{ \sg^{AB}}{2 (\sinh r)^{2}}\partial_r\kl((\sinh r)^{2}\sg_{BD}\kr)=\frac{1}{\tanh r}\delta^A_D,\\
\Gamma^{r}{}_{BD}=&-\frac{1}{2}\partial_r\kl((\sinh r)^{2}\sg_{BD}\kr)=-\frac{1}{\tanh r}g_{BD},\\
\Gamma^{A}{}_{BD}=&\sGamma^{A}{}_{BD},
\end{Eq*}
where $\sGamma$ denotes the Christoffel symbols of $\sg$ associated with the angular coordinates $\Theta$.

In particular, for $X=X^r\partial_r$, we have
\begin{Eq*}
\nabla_\alpha X^\alpha=\partial_r(X^r)+\Gamma^{A}{}_{rA}X^r=\hpartial_rX^r,
\end{Eq*}
where $\hpartial_r=\partial_r+2w$ with $w:=\frac{n-1}{2\tanh r}$, as defined earlier.
Thus, $\hpartial_r(\cdots)$ may be viewed as a divergence-form expression.
Moreover, it is straightforward to verify that $w$ satisfies
\begin{Eq}\label{Eq:cow}
w^2+w'-\frac{(n-1)^2}{4}=\frac{(n-1)(n-3)}{4\sinh^2 r},\qquad 
w''+2ww'=-\frac{2}{\tanh r}\frac{(n-1)(n-3)}{4\sinh^2 r}.
\end{Eq}

Next, we introduce a family of Killing vector fields $\{K_i\}_{i=1}^n$.
The first one is
\begin{Eq*}
K_1:=\cos\theta^1\partial_r-\sin\theta^1\coth r\partial_{\theta^1},
\end{Eq*}
which may be compared with the translation vector field $\partial_{x^1}=\cos\theta^1\partial_r-\sin\theta^1r^{-1}\partial_{\theta^1}$ in $\rR^n$.
The remaining Killing vector fields, $K_i$ for $i=2,3,\cdots,n$, are obtained by rotation and correspond to $\partial_{x^i}$ in $\rR^n$.

For this family of vector fields $\{K_i\}_{i=1}^n$, we have two key formulas:
\begin{Eq*}
\sum_{i=1}^n|K_if|^2=&|\partial_r f|^2+\frac{1}{\tanh^2 r}\sg^{AB}\partial_Af\partial_Bf=|\partial_r f|^2+\cosh^2 r|\spartial f|^2,\\
\sum_{i=1}^nK_i^2=&\partial_r^2+\frac{n-1}{\tanh r}\partial_r +\frac{1}{\tanh^2 r}\sg^{AB}\snabla_A\snabla_B
=\Delta_{\rH}+\Delta_{\rS}
=\Delta_{\rH}+\sum_{1\leq i<j\leq n}\Omega_{ij}^2,
\end{Eq*}
where $\{\Omega_{ij}\}_{1\leq i<j\leq n}$ are the standard rotation vector fields,
which are also Killing vector fields on $\rH^n$.
Additionally, we have
\begin{Eq*}{}
[K_i,K_j]=\Omega_{ij},\quad [K_i,\Omega_{jk}]=\delta_{ij}K_k-\delta_{ik}K_j.
\end{Eq*}
For convenience, we adjoin the time translation separately and introduce the notation
\begin{Eq}\label{Eq:Doz}
  Z\equiv\{\partial_t\}\cup\{K_i\}_{i=1}^n\cup\{\Omega_{ij}\}_{1\leq i<j\leq n}.
\end{Eq}
Then each element of $Z$ commutes with $\square_{\rH}$.

Next, we introduce a dyadic envelope for the perturbation coefficients. Suppose that a nonnegative quantity $\Pi=\Pi(t,r,\Theta)$ satisfies
\begin{Eq*}
\sum_{j=-\infty}^\infty \sup_{\rR\times A_j}\Pi\leq \varepsilon'.
\end{Eq*}
Without loss of generality, we may choose a summable sequence $\{\varepsilon'_j\}_{j=-\infty}^\infty$ such that
\begin{Eq*}
\sup_{\rR\times A_j}\Pi\leq \varepsilon'_j,\qquad 
\sum_{j=-\infty}^\infty \varepsilon'_j\approx \varepsilon'.
\end{Eq*}
Correspondingly, we define the piecewise constant function
\begin{Eq*}
\varepsilon'(r):=\varepsilon'_j,\qquad r\in[2^{j-\frac12},2^{j+\frac12}).
\end{Eq*}
Then, whenever $(r,\Theta)\in A_j$, we may simply write
\begin{Eq*}
\Pi(t,r,\Theta)\lesssim \varepsilon'(r).
\end{Eq*}
In the sequel, $\Pi$ will stand for the weighted combinations of perturbation terms appearing in \Hs{Hs:hoP} or in analogous hypotheses.

Finally, we introduce some useful inequalities.
\begin{proposition}\label{Pn:ki}
Let $h\in C^1([0,T])$ be a positive and increasing function.
Then, for any $f\in C^1([0, T])$ with $f(0)=0$, we have the inequality
\begin{Eq}\label{Eq:ki1}
\frac{|f(T)|^2}{h(T)}+\int_{0}^{T}|f|^2\frac{h'}{h^2}\d t\le 4\int_{0}^{T}\frac{|f'|^2}{h'}\d t.
\end{Eq}
On the other hand, let $T> 0$ and $h\in C^1([T,\infty))$ be a positive and increasing function.
Then, for any $f\in C^1([T, \infty))$ with $\lim_{t\to \infty}f=0$, we have the inequality
\begin{Eq}\label{Eq:ki2}
|f(T)|^2h(T)+\int_{T}^{\infty} |f|^2h'\d t\le 4 \int_{T}^{\infty}|f'|^2\frac{h^{2}}{h'}\d t.
\end{Eq}
\end{proposition}

Here, we note that it is not necessary to assume $h(0)\neq 0$ in \eqref{Eq:ki1}.
The proof is elementary and follows a similar approach to that of Hardy's inequality.
For the reader's convenience, we provide the proof below.

\begin{proof}
Firstly, we observe that \eqref{Eq:ki2} can be derived from \eqref{Eq:ki1}
by the change of variables $\tilde T=1/T$, $\tilde f(t)=f(1/t)$, and
$\tilde h(t)=\frac{1}{h(1/t)}$.
Thus, it suffices to prove \eqref{Eq:ki1}.

To begin, for any $t_1\in (0, T]$, we have
\begin{Eq*}
\int_{t_1}^{T}|f|^2\frac{h'}{h^2}\d t=- \left.\frac{|f|^2}{h}\right|_{t_1}^T
+\int_{t_1}^T (|f|^2)'\frac{1}{h}\d t
\le - \left.\frac{|f|^2}{h}\right|_{t_1}^T
+2\int_{t_1}^T |f||f'|\frac{1}{h}\d t,
\end{Eq*}
which implies that for any $\varepsilon'>0$,
\begin{Eq*}
 \frac{|f(T)|^2}{h(T)}+\int_{t_1}^{T}|f|^2\frac{h'}{h^2}\d t\le\frac{|f(t_1)|^2}{h(t_1)}
+\varepsilon'\int_{t_1}^T |f|^2\frac{h'}{h^2}\d t
+\frac{1}{\varepsilon'}\int_{t_1}^T \frac{|f'|^2}{h'}\d t.
\end{Eq*}
By choosing $\varepsilon'=1/2$,
the proof of
\eqref{Eq:ki1} is reduced to verifying the claim:
$\lim_{t_1\to 0}|f(t_1)|^2/h(t_1)=0$, which is nontrivial in the case when $h(0)=0$.

Recalling $f(0)=0$ and $h$ is increasing and positive, 
we obtain
\begin{Eq*}
\int_{0}^{t_1} |f|^2h'\d t
\leq |f(t_1)|^2h(t_1)+2\int_{0}^{t_1} |f||f'|h\d t
\leq |f(t_1)|^2h(t_1)+2h(t_1)\int_{0}^{t_1} |f||f'|\d t.
\end{Eq*}
Observing that
\begin{Eq*}
|f(t_1)|^2=|f(0)|^2+2\int_0^{t_1}ff'\d t
\leq 2\int_0^{t_1} |f||f'|\d t,
\end{Eq*}
we deduce from the previous inequality that
\begin{Eq*}
\int_{0}^{t_1} |f|^2h'\d t+|f(t_1)|^2h(t_1)\le 6h(t_1)\int_0^{t_1} |f||f'|\d t.
\end{Eq*}
By applying Young's inequality, we obtain
\begin{Eq*}
\int_{0}^{t_1} |f|^2h'\d t+|f(t_1)|^2h(t_1)
\leq 
\frac{9}{\varepsilon'}h^2(t_1)\int_0^{t_1}\frac{|f'|^2}{h'}\d t+\varepsilon'\int_0^{t_1} |f|^2h'\d t,
\end{Eq*}
for any $\varepsilon'>0$. Taking $\varepsilon'=1$, we conclude that
\begin{Eq*}
\frac{|f(t_1)|^2}{h(t_1)}
\leq 
9\int_0^{t_1}\frac{|f'|^2}{h'}\d t.
\end{Eq*}
As $t_1$ approaches $0$, we find $\lim_{t_1\to 0^+}|f(t_1)|^2/h(t_1)=0$, 
which completes the proof.
\end{proof}
\begin{remark}
This proposition plays a crucial role in the subsequent proofs.
Its most immediate application is the derivation of pointwise estimates from the energy of the solution.
In particular, if we take
$f=\sinh^{(n-1)/2}r\phi(r,\Theta)$ in \eqref{Eq:ki1}, we obtain
\begin{Eq}\label{Eq:hti}
\|r^{-1}\phi(r,\Theta)\|_{L^2(\rH^n)}^2+\sup_{R>0}
\frac{\sinh^{n-1} R }{R}\|\phi(R,\Theta)\|_{L^2(\rS^{n-1})}^2
\leq
4
\|\partial_r\phi+w\phi\|_{L^2(\rH^n)}^2\le 4\|\tpartial\phi\|_{L^2(\rH^n)}^2.
\end{Eq}
\end{remark}

\section{Global existence for quasilinear wave equations} 
\label{Sn:qlw}
\subsection{Iteration scheme and bootstrap quantities}
In this section, we prove \Tm{Tm:P1} based on \Tm{Tm:T}.
The proof proceeds as follows.
We set $\phi_{0}=0$ and consider the following iteration scheme:
\begin{Eq*}
\begin{cases}
\kl(\square_{\rH}-F_{1,k}\partial_t^2\kr)\phi_{k+1}=F_{2,k},\qquad 
F_{\iota,k}:=F_\iota(\partial_t^{\leq 1}\phi_k),\quad (\iota=1,2),\\
(\phi_{k+1},\partial_t\phi_{k+1})_{t=0}=\varepsilon\cdot(\psi_0,\psi_1),
\end{cases}
\end{Eq*}
where $F_1, F_2$ are smooth, $F_1(0,0)=F_2(0,0)=F_2'(0,0)=0$, and $\varepsilon>0$ is sufficiently small.
For the vector fields $Z$ defined by \eqref{Eq:Doz}, we define
\begin{Eq*}
\varE_{m,k}:=\|\phi_k\|_{\mathcal{E}_m},\qquad 
\varF_k:=\|\phi_{k+1}-\phi_k\|_{\mathcal{E}_0}.
\end{Eq*}
where $\|f\|_{\mathcal{E}_m}:=\|Z^{\leq m}f\|_{E\cap LE(S_0^\infty)}$.

The proof proceeds in three stages.
First, we show that the iterates are uniformly bounded in the higher-order norm $\mathcal{E}_{n+4}$, namely,
$\varE_{n+4,k}\lesssim \varepsilon$ uniformly in $k$.
Next, we prove that the iteration is contractive with respect to the base norm $\mathcal{E}_0$ by establishing
$\varF_{k+1}\le \frac12 \varF_k$.
Consequently, $\{\phi_k\}$ converges to a global finite-energy solution $\phi$ of \eqref{Eq:P1}.
Finally, we establish persistence of regularity for the limit $\phi$ by deriving higher-order estimates and applying Gronwall's inequality.

Before proceeding, we claim that
\begin{Eq}\label{Eq:coi3}
\|\tpartial^{\leq 1}Z^{\leq m}\phi_k(0)\|_{L^2}\lesssim\varepsilon,
\end{Eq}
for any $m\geq 0$, where the implicit constant is independent of $k$. We postpone the proof of this claim to \Sn{Sn:pocoi3}.
Moreover, by standard local well-posedness and finite propagation speed for strictly hyperbolic equations, each iterate $\phi_k$ is smooth and compactly supported on every time slice, since the initial data are compactly supported and the coefficients remain uniformly close to those of the background metric. Therefore, the compact-support assumption in \Tm{Tm:T} is satisfied at every step of the iteration, and the theorem may be applied to $\phi_{k+1}$ and its commuted equations without further comment.

\subsection{Uniform boundedness
}\label{Sn:mii}
First, we note that $\varE_{n+4,0}=0$. 
Next, for some fixed $k\geq 0$, we assume that 
$\varE_{n+4,k}\lesssim\varepsilon$.
We shall show that the same bound also holds for $\varE_{n+4,k+1}$.

Here, by \eqref{Eq:hti} and the Sobolev inequality on the sphere, we obtain
\begin{Eq}\label{Eq:coLi}
\kl\|r^{-\frac{1}{2}}\sinh^{\frac{n-1}{2}}rZ^{\leq n+4-\lfloor \frac{n+1}{2}\rfloor}\phi_k\kr\|_{L^\infty L^\infty}
\lesssim&\sup_{t,r}\kl\|r^{-\frac{1}{2}}\sinh^{\frac{n-1}{2}}rZ^{\leq n+4}\phi_k\kr\|_{L^2(\rS^{n-1})}\\
\lesssim&\kl\|\tpartial Z^{\leq n+4}\phi_k\kr\|_{L^\infty L^2}
\lesssim \varE_{n+4,k}
\lesssim\varepsilon.
\end{Eq}
Moreover, since the vector fields $Z$ are essentially adapted to spatial translations, 
and since the norm $\|\tpartial f\|_{L^\infty L^2}$ is also essentially invariant under spatial translations (see \Sn{Sn:eefuo}), 
we may remove the apparent singularity at $r=0$.

Since $F_1$ is smooth with $F_1(0,0)=0$, we easily see that
$|F_1(\vec{u})|\lesssim |\vec{u}|$ and $|\nabla F_1(\vec{u})|\lesssim 1$ as long as $|\vec{u}|\lesssim 1$.
On the other hand, by the previous estimate, we have
\begin{Eq*}
|Z^{\leq 2}\phi_k|
\lesssim \kl<r\kr>^{\frac12}\cosh^{-\frac{n-1}{2}}r\,\varE_{n+4,k}
\lesssim \varepsilon\cdot \cosh^{\delta_0-\frac{n-1}{2}}r.
\end{Eq*}
Therefore, for any $0<\delta_0\ll 1$, noting that $|\partial_r f|\lesssim |Kf|$, we have
\begin{Eq*}
|\partial_{t,r}^{\leq 1}F_{1,k}|
\lesssim |Z^{\leq 1}F_{1,k}|
\lesssim |\partial_t^{\leq 1}\phi_k|+|Z\partial_t^{\leq 1}\phi_k|
\lesssim \varepsilon\cdot \cosh^{\delta_0-\frac{n-1}{2}}r.
\end{Eq*}

Hence, if $\varepsilon$ is sufficiently small, then $\gamma^{tt}=-F_{1,k}$ and $G^t=\partial_tF_{1,k}$ satisfy the assumptions of \Hs{Hs:hoP}, 
with $\varepsilon'$ uniformly in $k$. Since all the other perturbation components vanish, we may apply \Tm{Tm:T}.

In addition, for any $0 < \delta_0, \delta_1 \ll 1$ and any sufficiently regular function $f$, 
we have
\begin{Eq}\label{Eq:sf}
\|r^{-\frac{1}{2}+\delta_0}\kl<r\kr>^{-\delta_0-\delta_1}f\|_{L^2(\rH^n)}\lesssim\sup_j 2^{-j/2}\|f\|_{L^2(A_j)},
\end{Eq}
where the implicit constant depends only on $\delta_0$ and $\delta_1$. 
The proof is straightforward and is left to the reader. 
Consequently, we also obtain
\begin{Eq*}
\kl\|r^{-\frac{1}{2}+\delta_0}\kl<r\kr>^{-\delta_0-\delta_1}(\tpartial,r^{-1})Z^{\leq n+4}\phi_k\kr\|_{L^2L^2}
\lesssim \kl\|Z^{\leq n+4}\phi_k\kr\|_{LE}
\lesssim \varE_{n+4,k}
\lesssim \varepsilon.
\end{Eq*}

At this point, we are ready to begin the iteration estimates.
For $Z^{\leq m}\phi_{k+1}$, we have
\begin{Eq*}
(\square_{\rH}-F_{1,k}\partial_t^2)\kl(Z^{\leq m}\phi_{k+1}\kr)
=[Z^{\leq m},F_{1,k}]\partial_t^2\phi_{k+1}+Z^{\leq m}F_{2,k}.
\end{Eq*}
Therefore, using \eqref{Eq:Pgei}, 
the initial-data estimate \eqref{Eq:coi3}, 
and the dual version of \eqref{Eq:sf}, we obtain
\begin{Eq*}
\varE_{n+4,k+1}\lesssim&\|\tpartial Z^{\leq n+4}\phi_{k+1}(0)\|_{L^2}
+\kl\|Z^{\leq n+4}F_{2,k}\kr\|_{LE^*}
+\|[Z^{\leq n+4},F_{1,k}]\partial_t^2\phi_{k+1}\|_{LE^*}\\
\lesssim&\varepsilon+I_{11}+I_{12}.
\end{Eq*}
where
\begin{Eq*}
I_{11}:=&\kl\|r^{\frac{1}{2}-\delta_0}\kl<r\kr>^{\delta_0+\delta_1}Z^{\leq n+4}F_{2,k}\kr\|_{L^2L^2},\\
I_{12}:=&
\kl\|r^{\frac{1}{2}-\delta_0}\kl<r\kr>^{\delta_0+\delta_1}[Z^{\leq n+4},F_{1,k}]\partial_t^2\phi_{k+1}\kr\|_{L^2L^2}.
\end{Eq*}

Using smoothness of $F_2$ and the vanishing conditions $F_2(0,0)=F_2'(0,0)=0$, 
$Z^{\leq m}F_{2,k}$ can be written as a linear combination of terms of the form
\begin{Eq*}
a(\partial_t^{\leq 1}\phi_k)\cdot
\kl(\partial_t^{\leq 1}Z^{\leq \lfloor \frac{m}{2}\rfloor}\phi_k\kr)^{i+1}
\cdot (\partial_t^{\leq 1}Z^{\leq m}\phi_k),
\end{Eq*}
for some $0\leq i\leq m-2$, where $a$ is essentially a smooth function arising from derivatives of $F_2$, thus $|a(\partial_t^{\leq 1}\phi_k)|\lesssim 1$ as long as $\varE_{n+4,k}\lesssim \varepsilon$.
Since $\partial_t\in{Z}$, under the above estimates, we obtain
\begin{Eq*}
I_{11}\lesssim&
\kl\|r^{-\frac{1}{2}+\delta_0}\kl<r\kr>^{-\delta_0-\delta_1-1}\cosh^{1-\delta_0} r\, Z^{\leq n+4}F_{2,k}\kr\|_{L^2L^2}
\kl\|r^{1-2\delta_0}\kl<r\kr>^{2\delta_0+2\delta_1+1}\cosh^{\delta_0-1}r\kr\|_{L_r^\infty}\\
\lesssim&
\sum_{i=0}^{n+2}
\kl\|r^{-\frac{1}{2}+\delta_0}\kl<r\kr>^{-\delta_0-\delta_1-1}\cosh^{1-\delta_0} r
\kl(Z^{\leq \lfloor \frac{n+4}{2}\rfloor+1}\phi_k\kr)^{i+1}\cdot\partial_t^{\leq 1} Z^{\leq n+4}\phi_k\kr\|_{L^2L^2}\\
\lesssim&
\sum_{i=0}^{n+2}
\kl\|\cosh^{\frac{1-\delta_0}{i+1}}rZ^{\leq \lfloor \frac{n+4}{2}\rfloor+1}\phi_k\kr\|_{L^\infty L^\infty}^{i+1}
\|r^{-\frac{1}{2}+\delta_0}\kl<r\kr>^{-\delta_0-\delta_1-1}\tpartial^{\leq 1}Z^{\leq n+4}\phi_k\|_{L^2L^2}\\
\lesssim&
\sum_{i=0}^{n+2}
\kl\|\cosh^{\frac{n-1}{2}-\delta_0}r Z^{\leq \lfloor \frac{n+4}{2}\rfloor+1}\phi_k\kr\|_{L^\infty L^\infty}^{i+1}
\|r^{-\frac{1}{2}+\delta_0}\kl<r\kr>^{-\delta_0-\delta_1}(\tpartial,r^{-1})Z^{\leq n+4}\phi_k\|_{L^2L^2}\\
\lesssim&(1+\varE_{n+4,k})^{n+2}(\varE_{n+4,k})^2
\lesssim \varepsilon \varE_{n+4,k}.
\end{Eq*}
Here, in the second-to-last inequality, we used \eqref{Eq:coLi}, together with
$\lfloor \frac{n+4}{2}\rfloor+1\leq n+4-\lfloor \frac{n+1}{2}\rfloor$.

Next, we turn to $I_{12}$. Using smoothness of $F_1$ and the vanishing condition $F_1(0,0)=0$, 
$[Z^{\leq m},F_{1,k}]\partial_t^2\phi_{k+1}$ 
can be written as a linear combination of terms of the form
\begin{Eq*}
a(\partial_t^{\leq 1}\phi_k)\cdot
\kl( \partial_t^{\leq 1} Z^{\leq \lfloor \frac{m+1}{2}\rfloor}\phi_k\kr)^{i+1}
\cdot
\partial_t^2 Z^{\leq m-1}\phi_{k+1},
\end{Eq*}
or
\begin{Eq*}
a(\partial_t^{\leq 1}\phi_k)\cdot
\kl( \partial_t^{\leq 1} Z^{\leq \lfloor \frac{m}{2}\rfloor}\phi_k\kr)^{i}
\cdot
\partial_t^{\leq 1} Z^{\leq m}\phi_k
\cdot\partial_t^2 Z^{\leq \lfloor \frac{m}{2}\rfloor-1}\phi_{k+1},
\end{Eq*}
for some $0\leq i\leq m-2$. 

Thus, arguing as in the estimate for $I_{11}$, we obtain
\begin{Eq*}
I_{12}\lesssim& \sum_{i=0}^{n+2}
\kl\|\cosh^{\frac{n-1}{2}-\delta_0}r Z^{\leq \lfloor \frac{n+5}{2}\rfloor+1}\phi_k\kr\|_{L^\infty L^\infty}^{i+1}
\|r^{-\frac{1}{2}+\delta_0}\kl<r\kr>^{-\delta_0-\delta_1}(\tpartial,r^{-1})Z^{\leq n+4}\phi_{k+1}\|_{L^2L^2}\\
&\ \ +\kl\|\cosh^{\frac{n-1}{2}-\delta_0}r Z^{\leq\lfloor \frac{n+4}{2}\rfloor+1}\phi_k\kr\|_{L^\infty L^\infty}^{i}
\|r^{-\frac{1}{2}+\delta_0}\kl<r\kr>^{-\delta_0-\delta_1}(\tpartial,r^{-1})Z^{\leq n+4}\phi_{k}\|_{L^2L^2}\\
&\ \ \ \ \ \times \kl\|\cosh^{\frac{n-1}{2}-\delta_0}r Z^{\leq \lfloor \frac{n+4}{2}\rfloor+1}\phi_{k+1}\kr\|_{L^\infty L^\infty}\\
\lesssim&(1+\varE_{n+4,k})^{n+2}\cdot\varE_{n+4,k}\cdot \varE_{n+4,k+1}\lesssim \varepsilon \varE_{n+4,k+1}
\end{Eq*}
since $\lfloor \frac{n+5}{2}\rfloor+1\leq n+4-\lfloor \frac{n+1}{2}\rfloor$.
Summarizing the above discussion, we conclude that, for some constant $C>0$,
\begin{Eq*}
\varE_{n+4,k+1}\leq C(\varepsilon+\varepsilon\varE_{n+4,k}+\varepsilon\varE_{n+4,k+1}).
\end{Eq*}
It then follows that $\varE_{n+4,k+1}\leq 2C\varepsilon$, 
provided that $\varE_{n+4,k}\leq 2C\varepsilon$ and $\varepsilon\leq (4C)^{-1}$.
This completes the proof of the uniform boundedness of the iteration.

\subsection{Contraction of the iteration and regularity of the limit}\label{Sn:car}
Next, turning to $\varF_k$, we compute that
\begin{Eq*}
\begin{cases}
\kl(\square_{\rH}-F_{1,k}\partial_t^2\kr)\kl(\phi_{k+1}-\phi_k\kr)=F_{2,k}-F_{2,k-1}+(F_{1,k}-F_{1,k-1})\partial_t^2\phi_k,\\
\kl(\phi_{k+1}-\phi_k,\partial_t\phi_{k+1}-\partial_t\phi_k\kr)_{t=0}=(0,0).
\end{cases}
\end{Eq*}
As before, by the mean value theorem, the right-hand side can be written as a linear combination of terms of the form
\begin{Eq*}
a(\partial_t^{\leq 1}\phi_k,\partial_t^{\leq 1}\phi_{k-1})\cdot \partial_t^{\leq 2}\kl(\phi_k,\phi_{k-1}\kr)\cdot \partial_t^{\leq 1}(\phi_k-\phi_{k-1}),
\end{Eq*}
where $a$ is smooth and uniformly bounded, thanks to the pointwise estimates obtained in the previous subsection.
Therefore, arguing similarly as in the estimates of $I_{11}$ and $I_{12}$, with one factor replaced by
$\partial_t^{\le 1}(\phi_k-\phi_{k-1})$,
we obtain $\varF_{k}\lesssim(\varE_{n+4,k}+\varE_{n+4,k-1})\varF_{k-1}$.
Since $\varE_{n+4,k}\lesssim \varepsilon$ uniformly in $k$, we have $\varF_{k}\leq C'\varepsilon\varF_{k-1}$.
Therefore, provided that $\varepsilon\leq (2C')^{-1}$ and that the conditions in the previous subsection are satisfied, the iteration is contractive.
Consequently, $\{\phi_k\}$ converges strongly in $\mathcal{E}_0$ and weakly in $\mathcal{E}_{n+4}$ to a global solution $\phi$ of \eqref{Eq:P1} with $\|\phi\|_{\mathcal{E}_{n+4}}\lesssim\varepsilon$; see, e.g., \cite[Lemma 4.4]{MR4662297} for a more detailed discussion of a similar argument.

Having established convergence of the iteration, 
we now turn to the regularity of the limit $\phi$.
For any fixed $T > 0$, we consider $t \in [0, T)$ and begin by computing
\begin{Eq*}
\varE_{m}(t):=&\|\tpartial^{\leq 1}Z^{\leq m}\phi\|_{L^\infty L^2(S_0^t)}\\
\lesssim&\|\tpartial Z^{\leq m}\phi\|_{L^\infty L^2(S_0^t)}+\|Z^{\leq m}\phi(0)\|_{L^2(\rH^n)}+T\|\partial_tZ^{\leq m}\phi\|_{L^\infty L^2(S_0^t)}\\
\lesssim&\|Z^{\leq m}\phi\|_{E(S_0^t)}+\|\tpartial^{\leq 1}Z^{\leq m}\phi(0)\|_{L^2(\rH^n)}.
\end{Eq*}
First, by the convergence of $\phi_k$ to $\phi$ and the uniform bound $\varE_{n+4,k}\lesssim\varepsilon$, we have
$\sup_{t\in[0,T)}\varE_{n+4}(t)\leq\varlimsup_{k\to\infty}\varE_{n+4,k}<\infty$.
Then, assuming that $\sup_{t\in[0,T)}\varE_{m-1}(t)<\infty$ for some fixed $m\geq n+5$, 
and arguing as in the previous subsection, we obtain
\begin{Eq*}
\varE_{m}(t)\lesssim&\varepsilon +\|Z^{\leq m}F_2(\phi,\partial_t\phi)\|_{L^1L^2(S_0^t)}+\|[Z^{\leq m},F_1(\phi,\partial_t\phi)\partial_t^2]\phi\|_{L^1L^2(S_0^t)}\\
\lesssim &1+\sum_{i=0}^{m-2}\kl\|\kl(Z^{\leq \lfloor \frac{m+1}{2}\rfloor+1}\phi\kr)^{i+1}\cdot \partial_t^{\leq 1}Z^{\leq m}\phi\kr\|_{L^1L^2(S_0^t)}\\
\lesssim &1+\kl(1+\kl\|Z^{\leq \lfloor \frac{m+1}{2}\rfloor+1}\phi\kr\|_{L^\infty L^\infty(S_0^t)}\kr)^{m-1}\kl\|\tpartial^{\leq 1}Z^{\leq m}\phi\kr\|_{L^1 L^2(S_0^t)}\\
\lesssim &1+\big(1+\varE_{m-1}(t)\big)^{m-1}\int_0^t\kl\|\tpartial^{\leq 1}Z^{\leq m}\phi\kr\|_{L^2(\rH^n)}\d\tau\\
\lesssim &1+\int_0^t \varE_{m}(\tau)\d\tau
\end{Eq*}
since $\lfloor \frac{m+1}{2}\rfloor+1\leq m-1-\lfloor \frac{n+1}{2}\rfloor$ for $m\geq n+5$.

Combining these results, we obtain
\begin{Eq*}
\varE_{m}(t)\lesssim 1+\int_0^t\varE_{m}(\tau)\d\tau.
\end{Eq*}
By Gronwall's inequality, it follows that $\sup_{t\in(0,T]}\varE_{m}(t)<\infty$.
By induction on $m$, we conclude that this bound holds for all $m$, 
which implies that $\phi \in C^\infty(S_0^T)$.
Since $T>0$ is arbitrary, this completes the proof of \Tm{Tm:P1}.

\subsection{Proof of Claim~(\ref{Eq:coi3})}\label{Sn:pocoi3}
In this subsection, we prove Claim \eqref{Eq:coi3}.
Observe that the full time jet of $\phi_k$ at $t=0$ is determined recursively by the common initial data $(\psi_0,\psi_1)$ and the nonlinear structure of the equation, and is therefore independent of $k$. For this reason, we omit the subscript $k$ throughout this subsection.

To begin with, we show that the initial value of $\partial_t^m \phi$ takes the form
\begin{Eq*}
\kl.\partial_t^m\phi\kr|_{t=0}=f_m(\varepsilon;r,\Theta),
\end{Eq*}
where $f_m(\varepsilon;r,\Theta)\in C^\infty(\rR;C_0^\infty(\rH^n))$ and satisfies $f_m(0;\cdot,\cdot)=0$ for every $m$.

This is clear for $m=0,1$, since $f_0=\varepsilon\psi_0$ and $f_1=\varepsilon\psi_1$.
Now assume that the statement holds for all $m\le M-1$ for some fixed $M\ge2$.
From \eqref{Eq:P1}, we have
\begin{Eq*}
\partial_t^{M}\phi
=\partial_t^{M-2}\Big((1+F_1(\phi,\partial_t\phi))^{-1}\big(\Delta_{\rH}\phi+\rho^2\phi-F_2(\phi,\partial_t\phi)\big)\Big).
\end{Eq*}
Observe that the right-hand side involves only smooth algebraic combinations of
$\partial_t^i\phi$ for $i\leq M-1$, together with spatial derivatives acting on those quantities.
Therefore, after restricting to $t=0$, we may express $f_M$ in terms of smooth spatial derivatives of $f_m$ with $m<M$.
It follows that $f_M(\varepsilon;r,\Theta)\in C^\infty(\rR;C_0^\infty(\rH^n))$, and moreover $f_M(0;\cdot,\cdot)=0$.

Now, let $\hat Z = \{K_i\}_{i=1}^n \cup \{\Omega_{ij}\}_{1\leq i<j\leq n}
= Z \backslash \{\partial_t\}$.
Using the discussions in \Sn{Sn:p}, together with \eqref{Eq:eoe}, we obtain
\begin{Eq*}
\|\tpartial^{\leq 1}Z^{\leq m}\phi(0)\|_{L^2}^2\approx& \|\partial_t^{\leq 1}Z^{\leq m}\phi(0)\|_{L^2}^2+\kl(\kl(-\Delta_{\rH}-\rho^2\kr)Z^{\leq m}\phi(0),Z^{\leq m}\phi(0)\kr)_{\rH^n}\\
\lesssim& \sum_{i=0}^{m+1}\|\hat Z^{\leq m+2}\partial_t^{i}\phi(0)\|_{L^2}^2\\
=&\sum_{i=0}^{m+1}\|\hat Z^{\leq m+2}f_i(\varepsilon;r,\Theta)\|_{L^2(\rH^n)}^2\leq C_m\varepsilon^2.
\end{Eq*}
This concludes the proof of \eqref{Eq:coi3}.

\section{Energy and local energy estimates for the unperturbed wave operator}\label{Sn:energyLE}
In this section, we present the proof of \Tm{Tm:T} in the unperturbed setting.
Let $T_0$ denote the energy-momentum tensor associated with  $f$, defined by
\begin{Eq*}
T_{0}^{\alpha\beta}:=\nabla^\alpha f\nabla^\beta f-\frac{1}{2}g^{\alpha\beta}\kl(\nabla_\sigma f\nabla^\sigma f-\frac{(n-1)^2}{4}f^2\kr).
\end{Eq*}
Then, the divergence of $T_0$ satisfies $\nabla_\alpha T_{0}^{\alpha\beta}=-\square_{\rH}f\cdot\nabla^\beta f$. 

For any given triplet $(X,\chi,Y)$ where  $X$ and $Y$ are vector fields and $\chi$ is a function, 
we define the current
\begin{Eq*}
J_0^\alpha:=-T_{0}^{\alpha\beta}X_\beta+\frac{1}{2}f^2\nabla^\alpha\chi-\chi f\nabla^\alpha f+Y^\alpha.
\end{Eq*}
Next, we compute its divergence as follows:
\begin{Eq}\label{Eq:DoJ0}
\nabla_\alpha J_0^\alpha=&\square_{\rH}f Xf-T_{0}^{\alpha\beta}\pi_{\alpha\beta}+\frac{1}{2}f^2\nabla^\alpha\nabla_\alpha\chi-\chi\nabla_\alpha f\nabla^\alpha f-\chi f\nabla^\alpha\nabla_\alpha f+\nabla_\alpha Y^\alpha\\
=&\square_{\rH}f (X f+\chi f)-T_{0}^{\alpha\beta}\pi_{\alpha\beta}-\chi\nabla_\alpha f\nabla^\alpha f+\frac{(n-1)^2}{4}\chi f^2+\frac{1}{2}f^2\nabla_\alpha\nabla^\alpha\chi+\nabla_\alpha Y^\alpha.
\end{Eq}
Here, $\pi$ denotes the deformation tensor associated with $X$, given by
\begin{Eq*}
\pi_{\alpha\beta}=\frac{1}{2}(\nabla_\alpha X_\beta+\nabla_\beta X_\alpha)=\frac{1}{2}\kl(g_{\beta\sigma}(\partial_\alpha X^\sigma+\Gamma^\sigma{}_{\alpha\delta}X^\delta)+g_{\alpha\sigma}(\partial_\beta X^\sigma+\Gamma^\sigma{}_{\beta\delta}X^\delta)\kr).
\end{Eq*}

For convenience, we omit the volume element in integrals throughout this paper unless ambiguity arises.
By integrating $\nabla_\alpha J_0^\alpha$ over $S_0^T$ and applying Gauss' theorem, 
we obtain
\begin{Eq}\label{Eq:GioJ0}
\kl(\int_{\rH^n}J_0^t\kr)_{t=0}^T=\int_{S_0^T}\nabla_\alpha J_0^\alpha.
\end{Eq}
With this foundation established, 
we first derive the energy estimate for the unperturbed wave operator $\square_{\rH}$.
Subsequently, we extend our analysis to obtain its local energy estimate.

\subsection{Energy estimate}\label{Sn:eefuo} 
In this subsection, we fix $(X,\chi,Y)=(\partial_t,0,0)$, which yields
\begin{Eq*}
J_0^t=&T_0^{tt}=\frac{1}{2}\kl(|\partial_t f|^2+(\partial_r f)^2+|\spartial f|^2-\frac{(n-1)^2}{4} f^2\kr).
\end{Eq*}
Unlike in the Euclidean case, 
this expression is not automatically non-negative.
However, by using \eqref{Eq:cow}, 
we obtain the refined form
\begin{Eq*}
J_0^t=&\frac{1}{2}\kl(\kl|\partial_t f\kr|^2+|\partial_r f+w f|^2+|\spartial f|^2\kr)+\frac{1}{2}\kl(w^2+w'-\frac{(n-1)^2}{4}\kr) f^2-\frac{1}{2}\hpartial_r\kl(w f^2\kr)\\
=&\frac{1}{2}\kl(\kl|\partial_t f\kr|^2+|\partial_r f+w f|^2+|\spartial f|^2+\frac{(n-1)(n-3)}{4\sinh^2 r} f^2\kr)-\frac{1}{2}\hpartial_r\kl(w f^2\kr).
\end{Eq*}
Since $n\geq 3$, the leading term is non-negative.
Moreover, since $\hpartial_r$ is a radial divergence operator and $f$ is compactly supported on each time slice, 
the integral of the last term over $\rH^n$ vanishes.
Therefore, after integration over $\rH^n$, the quantity $\int_{\rH^n} J_0^t$ is coercive.

On the other hand, since $\partial_t$ is a Killing vector field for the product Lorentzian metric, 
its deformation tensor vanishes, namely $\pi=0$. 
Using \eqref{Eq:DoJ0}, we compute
\begin{Eq*}
\nabla_\alpha J_0^\alpha=&\square_{\rH} f\cdot\partial_t f.
\end{Eq*}
From \eqref{Eq:GioJ0}, we then obtain
\begin{Eq*}
&\kl(\int_{\rH^n}|\tpartial f|^2\kr)_{t=T}\lesssim \|\tpartial f(0)\|_{L^2(\rH^n)}^2+\int_{S_0^T}\kl|\square_{\rH} f\kr|\cdot\kl|\partial_t f\kr|,
\end{Eq*}
and hence
\begin{Eq*}
\|f\|_{E(S_0^T)}\lesssim \|\tpartial f(0)\|_{L^2(\rH^n)}+\|\square_{\rH} f\|_{L^1_tL^2_x(S_0^T)}.
\end{Eq*}
Combining \eqref{Eq:hti} with the above coercive decomposition, 
we obtain the energy part of \Tm{Tm:T} for the unperturbed operator.

Additionally, the above argument also yields a norm equivalence:
for any sufficiently regular $f(t,r,\Theta)$, we have
\begin{Eq}\label{Eq:eoe}
\kl\|\tpartial f\kr\|_{L^2}^2
\approx \|\partial_tf\|_{L^2}^2+\kl\|\sqrt{-\Delta_{\rH}-\rho^2}f\kr\|_{L^2}^2.
\end{Eq}

\subsection{Local energy estimate}\label{Sn:leefuo} 
In this subsection, we set the triplet as $(X,\chi,Y)=\kl(h_R\partial_r,h_Rw,-\frac{1}{2}h_R'\tw f^2 \partial_r\kr)$,
where $h_R=\frac{r}{R+r}$ and $\tw=\tw(r)$,
which will be specified later.
Through computation, we obtain:
\begin{Eq*}
J_0^t=&-h_RT_0^{tr}+h_R w f\partial_t f=h_R\partial_t f(\partial_r f+w f ).
\end{Eq*}
Observing that  $|h_R|\leq 1$, we immediately deduce that $\kl|\int_{\rH^n}J_0^t\kr|\lesssim\|\tpartial f\|_{L^2(\rH^n)}^2$. 
This ensures that the boundary integral in \eqref{Eq:GioJ0} 
can be absorbed by the previously established energy.

On the other hand, to analyze the inner integral in \eqref{Eq:GioJ0},
we first compute the nonzero components of $\pi$ as
\begin{Eq*}
\pi_{rr}=\partial_r X^r=h_R',\quad\pi_{AB}=\frac{g_{AD}\Gamma^{D}{}_{Br}+g_{BD}\Gamma^{D}{}_{Ar}}{2}X^r=\frac{h_R}{\tanh r}g_{AB},
\end{Eq*}
which lead to the following expressions:
\begin{Eq*}
T_0^{rr}\pi_{rr}=&\frac{h_R'}{2}\kl(|\partial_t f|^2+(\partial_r f)^2-|\spartial f|^2+\frac{(n-1)^2}{4} f^2\kr)\\
T_0^{AB}\pi_{AB}=&\frac{h_R}{\tanh r}|\spartial f|^2-h_Rw\kl(\nabla_\sigma f\nabla^\sigma f-\frac{(n-1)^2}{4} f^2\kr).
\end{Eq*}
As a result, \eqref{Eq:DoJ0} simplifies to
\begin{Eq*}
\nabla_\alpha J^\alpha
=&h_R\square_{\rH} f \cdot(\partial_r f+w f)-T_0^{rr}\pi_{rr}-T_0^{AB}\pi_{AB}-h_Rw\kl(\nabla_\sigma f\nabla^\sigma f-\frac{(n-1)^2}{4} f^2\kr)\\
&\quad+\frac{1}{2}\hpartial_r\partial_r(h_Rw) f^2-\frac{1}{2}\hpartial_r\kl(h_R'\tw f^2\kr)\\
=&h_R\square_{\rH} f \cdot\kl(\partial_r f+w f\kr)
-\frac{h_R'}{2}\kl(|\partial_t f|^2+(\partial_r f)^2-|\spartial f|^2+\frac{(n-1)^2}{4} f^2\kr)-\frac{h_R}{\tanh r}|\spartial f|^2\\
&\quad+\frac{1}{2}\hpartial_r\partial_r(h_Rw) f^2-\frac{1}{2}\hpartial_r\kl(h_R'\tw f^2\kr)\\
=&h_R\square_{\rH} f \cdot\kl(\partial_r f+w f\kr)-\frac{h_R'}{2}\kl(|\partial_t f|^2+|\partial_r f+\tw f|^2\kr)-\kl(\frac{h_R}{\tanh r}-\frac{h_R'}{2}\kr)|\spartial f|^2-\frac{1}{2}W f^2,\\
W:=&h_R''(\tw-w)+h_R'\kl(-2w'+\tw'-2w^2+2w\tw-\tw^2+\frac{(n-1)^2}{4}\kr)-h_R(w''+2ww').
\end{Eq*}
Among these terms, we analyze
\begin{Eq*}
\frac{h_R}{\tanh r}-\frac{h_R'}{2}=&\frac{h_R'}{2}+\frac{R(r-\tanh r)+r^2}{(R+r)^2\tanh r}\gtrsim \frac{h_R}{\tanh r}\gtrsim h_R'\geq 0,
\end{Eq*}
which shows that the coefficients of the derivative terms are all negative definite and comparable to $h_R'$.

For the $f^2$ term, we first set $\tw=w$ and simplify $W$. 
By applying  \eqref{Eq:cow}, we obtain
\begin{Eq*}
W=&h_R'\kl(-w'-w^2+\frac{(n-1)^2}{4}\kr)-h_R(w''+2ww')\\
=&\frac{(n-1)(n-3)}{2\sinh^2 r}\kl(\frac{h_R}{\tanh r}-\frac{h_R'}{2}\kr)\\
\gtrsim&(n-3) \frac{h_R}{\sinh^{2}r\tanh r}.
\end{Eq*}
Now, incorporating these results into \eqref{Eq:GioJ0}, we get
\begin{Eq*}
&\int_{S_0^T}h_R'|\tpartial f|^2+(n-3) \frac{h_R}{\sinh^{2}r\tanh r}|f|^2\lesssim \|f\|_{E(S_0^T)}^2+\int_{S_0^T}h_R\kl|\square_{\rH} f\kr|\cdot\kl|\partial_r f+w f\kr|.
\end{Eq*}
Next, we refine our choice by setting $\tw=w-\frac{1}{r}$. 
Applying \eqref{Eq:cow} once again, we derive
\begin{Eq*}
W=&-\frac{1}{r}h_R''+h_R'\kl(-w'-w^2+\frac{(n-1)^2}{4}\kr)-h_R(w''+2ww')\\
\geq&\frac{1}{r}|h_R''|=\frac{2R}{r(r+R)^3}\geq 0.
\end{Eq*}
Finally, substituting this into \eqref{Eq:GioJ0} and combining it with the previous result, 
we obtain
\begin{Eq*}
&\int_{S_0^T}h_R'|\tpartial f|^2+\kl((n-3) \frac{h_R}{\sinh^{2}r\tanh r}+r^{-1}|h_R''|\kr)|f|^2\lesssim \|f\|_{E(S_0^T)}^2+\int_{S_0^T}h_R\kl|\square_{\rH} f\kr|\cdot\kl|\partial_r f+w f\kr|.
\end{Eq*}

Finally, for any fixed $j$ with $r\approx R\approx 2^j$, we observe that
\begin{Eq*}
h_R'\approx 2^{-j},\qquad r^{-1}|h_R''|\approx 2^{-3j}.
\end{Eq*}
Substituting these into the previous inequality 
and combining the result with the energy estimate obtained in \Sn{Sn:eefuo}, 
we derive
\begin{Eq}\label{Eq:leefuo}
&\int_{S_0^T}h_R'|\tpartial f|^2+(n-3) \frac{h_R}{\sinh^{2}r\tanh r}|f|^2+\| f\|_{E\cap LE(S_0^T)}^2\\
\lesssim&\int_{S_0^T}|\square_{\rH} f|\kl|\kl(\partial_r f+w f,\partial_t f\kr)\kr|+\|\tpartial f(0)\|_{L^2(\rH^n)}^2\\
\lesssim&\kl\|\square_{\rH} f\kr\|_{[L^1L^2+LE^*](S_0^T)}\cdot\| f\|_{[E\cap LE](S_0^T)}+\|\tpartial f(0)\|_{L^2(\rH^n)}^2.
\end{Eq}
This completes the proof of \Tm{Tm:T} for the unperturbed operator.

\section{Energy and local energy estimates for the perturbed wave operator}\label{Sn:efpo} 
In this section, we present the proof of \Tm{Tm:T} in the perturbed setting.
To analyze the perturbed wave operator, 
we first examine the principal part of the perturbation.
We introduce the operator $\square_\gamma$ and define the associated energy-momentum tensor and current for $f$ as follows:
\begin{Eq*}
\square_\gamma:=&\nabla_\alpha\kl(\gamma^{\alpha\beta}\nabla_\beta\kr),\\
T_\gamma^{\alpha\beta}:=&\gamma^{\alpha \sigma}\nabla_\sigma f\nabla^\beta f-\frac{1}{2}g^{\alpha \beta}\gamma^{\sigma \delta}\nabla_\sigma  f\nabla_\delta f,\\
J_\gamma^\alpha:=&-T_\gamma^{\alpha\beta}X_\beta+\frac{1}{2}\gamma^{\alpha \beta}\nabla_\beta\chi f^2-\gamma^{\alpha\beta}\chi f\nabla_\beta f+Y^\alpha,
\end{Eq*}
for any given triplet $(X,\chi,Y)$.
Through detailed calculations, we derive the divergence of the current:
\begin{Eq}\label{Eq:DoJg}
\nabla_\alpha J_\gamma^\alpha =&-\nabla_\alpha \kl(\gamma^{\alpha \sigma}\nabla_\sigma  f\nabla^\beta  f-\frac{1}{2}g^{\alpha \beta}\gamma^{\sigma \delta}\nabla_\sigma  f\nabla_\delta  f\kr)X_\beta -T_\gamma^{\alpha \beta}\nabla_\alpha X_\beta \\
&\quad+\nabla_\alpha \kl(\frac{1}{2}\gamma^{\alpha \beta}\nabla_\beta \chi f^2-\gamma^{\alpha \beta}\chi f\nabla_\beta  f+Y^\alpha \kr)\\
=&-\nabla_\alpha \kl(\gamma^{\alpha \sigma}\nabla_\sigma  f\kr)X f +\frac{1}{2}X^\alpha\nabla_\alpha\gamma^{\sigma \delta}\nabla_\sigma  f\nabla_\delta  f -T_\gamma^{\alpha \beta}\nabla_\alpha X_\beta +\frac{1}{2}\nabla_\alpha \kl(\gamma^{\alpha \beta}\nabla_\beta \chi\kr) f^2\\
&\quad -\nabla_\alpha (\gamma^{\alpha \beta}\nabla_\beta  f)\chi f-\chi\gamma^{\alpha \beta}\nabla_\alpha  f\nabla_\beta  f+\nabla_\alpha Y^\alpha \\
=&-\square_\gamma f( X f+\chi f)+\frac{1}{2}X^\alpha \nabla_\alpha \gamma^{\sigma \delta}\nabla_\sigma  f\nabla_\delta  f -T_\gamma^{\alpha \beta}\nabla_\alpha X_\beta 
-\chi\gamma^{\alpha \beta}\nabla_\alpha  f\nabla_\beta  f\\
&\quad+\frac{1}{2}\nabla_\alpha \kl(\gamma^{\alpha \beta}\nabla_\beta \chi\kr) f^2+\nabla_\alpha Y^\alpha.
\end{Eq}

Integrating $\nabla_\alpha (J_0^\alpha+J_\gamma^\alpha)$ over $S_0^T$
and applying Gauss' theorem,
we obtain
\begin{Eq}\label{Eq:GioJg}
\kl(\int_{\rH^n}(J_0^t+J_\gamma^t)\kr)_{t=0}^T=\int_{S_0^T}\nabla_\alpha (J_0^\alpha+ J_\gamma^\alpha).
\end{Eq}
Based on this integral, 
we first establish the energy estimate for the principal part of the perturbed wave operator, 
$\square_{\rH}+\square_{\gamma}$.
We then extend our analysis to derive its local energy estimate.
Finally, we examine the remaining perturbations within $P$ and provide a comprehensive concluding discussion.

\subsection{Energy estimate}
To ensure consistency with our previous estimates for $\tpartial f$,
we take$
(X,\chi,Y)=(\partial_t,0,Y^t\partial_t+Y^r\partial_r+Y^A\partial_A)$,
where the components of $Y$ are given by
\begin{Eq*}
Y^t=&-\frac{1}{2}\partial_r\gamma^{rr} w f^2-\frac{(n-1)^2}{8}\gamma^{rr} f^2-\frac{1}{2}\snabla_A\gamma^{rA}w f^2,\\
Y^r=&\frac{1}{2}\partial_t\gamma^{rr}w f^2,\\
Y^A=&\frac{1}{2}\partial_t\gamma^{rA}w f^2.
\end{Eq*}
With this setup, we derive the expression for $J_\gamma^t$:
\begin{Eq*}
J_\gamma^t=&T_\gamma^{tt}+Y^t=-\gamma^{t\sigma}\nabla_\sigma f\partial_t f+\frac{1}{2}\gamma^{\sigma\delta}\nabla_\sigma f\nabla_\delta f+Y^t\\
=&\kl(-\frac{1}{2}\gamma^{tt}|\partial_t f|^2+\frac{1}{2}\gamma^{AB} \partial_A f\partial_B f\kr)
+\frac{1}{2}\kl(\gamma^{rr}(\partial_r f)^2-\partial_r\gamma^{rr}w f^2-\frac{(n-1)^2}{4}\gamma^{rr} f^2\kr)\\
&\quad+\kl(\gamma^{rA}\partial_r f \partial_A f-\frac{1}{2}\snabla_A\gamma^{rA}w f^2\kr)\\
\equiv& I_{21}+\frac{1}{2}I_{22}+I_{23}.
\end{Eq*}
Following the assumptions in \Hs{Hs:hoP}, 
we immediately observe that $|I_{21}|\lesssim\varepsilon' |\tpartial f|^2$.
For $I_{22}$, by reformulating and applying \eqref{Eq:cow}, we obtain:
\begin{Eq*}
I_{22}=&\gamma^{rr}|\partial_r f+w f|^2+\gamma^{rr}\kl(w^2+w'-\frac{(n-1)^2}{4}\kr) f^2-\hpartial_r(\gamma^{rr}w f^2)\\
=&\gamma^{rr}|\partial_r f+w f|^2+\frac{(n-1)(n-3)}{4\sinh^2 r}\gamma^{rr} f^2-\hpartial_r(\gamma^{rr}w f^2)\\
\equiv& I_{22}'-\hpartial_r(\gamma^{rr}w f^2).
\end{Eq*}
From \Hs{Hs:hoP}, it follows that
$|I_{22}'|\lesssim\varepsilon'\cdot(|\tpartial f|^2+r^{-2}| f|^2)$.
Similarly, for $I_{23}$, we compute:
\begin{Eq*}
I_{23}
=&\gamma^{rA}\kl(\partial_r f+w f\kr) \partial_A f-\frac{1}{2}\snabla_A\kl(\gamma^{rA}w f^2\kr)
\equiv I_{23}'-\frac{1}{2}\snabla_A\kl(\gamma^{rA}w f^2\kr).
\end{Eq*}
Again, invoking \Hs{Hs:hoP}, we deduce that $|I_{23}'|\lesssim \varepsilon' |\tpartial f|^2$.
Since the remaining terms in $J_\gamma^t$ are divergence terms, 
we arrive at the following estimate:
\begin{Eq*}
\int_{\rH^n}(J_0^t+J_\gamma^t)\lesssim \int_{\rH^n}J_0^t+\varepsilon'\kl(|\tpartial f|^2+r^{-2}| f|^2\kr)\lesssim \int_{\rH^n}J_0^t.
\end{Eq*}
This confirms that the boundary integral in \eqref{Eq:GioJg} can be absorbed into the previously established energy.

On the other hand, noting that $\nabla_\alpha X_\beta=0$, 
equation \eqref{Eq:DoJg} simplifies to
\begin{Eq*}
\nabla_\alpha J_\gamma^\alpha
=&-\square_\gamma f\partial_t f
+\frac{1}{2}\partial_t\gamma^{\sigma\delta}\nabla_\sigma f\nabla_\delta f
+\partial_tY^t+\hpartial_rY^r+\snabla_A Y^A\\
=&-\square_\gamma f\partial_t f
+\kl(\frac{1}{2}\partial_t\gamma^{tt}|\partial_t f|^2
+\partial_t\gamma^{tA}\partial_t f\partial_A f
+\frac{1}{2}\partial_t\gamma^{AB}\partial_{A} f\partial_B f\kr)
+\partial_t\gamma^{tr}\partial_t f\partial_r f\\
&\quad+\kl(\partial_t\gamma^{rA}\partial_r f\partial_A f
+\partial_t\kl(-\frac{1}{2}\snabla_A\gamma^{rA}w f^2\kr)
+\snabla_A\kl(\frac{1}{2}\partial_t\gamma^{rA}w f^2\kr)\kr)\\
&\quad+\frac{1}{2}\kl(\partial_t\gamma^{rr}(\partial_r f)^2
+\partial_t\kl(-\partial_r\gamma^{rr} w f^2-\frac{(n-1)^2}{4}\gamma^{rr} f^2\kr)
+\hpartial_r\kl(\partial_t\gamma^{rr}w f^2\kr)\kr)\\
\equiv&-\square_\gamma f\partial_t f+I_{31}+I_{32}+I_{33}+\frac{1}{2}I_{34}.
\end{Eq*}
First, applying \Hs{Hs:hoP}, we derive
\begin{Eq*}
|I_{31}|\lesssim& \varepsilon'(r)r^{-1}|\tpartial f|^2.
\end{Eq*}
Next, for $I_{32}$, utilizing \Hs{Hs:hoP}, we conclude
\begin{Eq*}
I_{32}=&\partial_t\gamma^{tr}\partial_t f\kl(\partial_r f+w f\kr)-\partial_t\gamma^{tr} w f\partial_t f\\
|I_{32}|\lesssim& \varepsilon'(r)r^{-1}\kl(|\partial_t f|^2+|\partial_r f+w f|^2+\frac{| f|^2}{r^2}\kr)\lesssim\varepsilon'(r)r^{-1}(|\tpartial f|^2+r^{-2}| f|^2).
\end{Eq*}
Then, for $I_{33}$, employing \Hs{Hs:hoP}, we obtain
\begin{Eq*}
I_{33}=&\partial_t\gamma^{rA}(\partial_r f+w f)\partial_A f-\snabla_A\gamma^{rA}w f\partial_t f\\
|I_{33}|\lesssim&\varepsilon'(r)r^{-1}\kl(|\partial_t f|^2+|\partial_r f+w f|^2+|\spartial f|^2+\frac{| f|^2}{r^2}\kr)\lesssim\varepsilon'(r)r^{-1}(|\tpartial f|^2+r^{-2}| f|^2).
\end{Eq*}
Finally, for $I_{34}$, leveraging both \Hs{Hs:hoP} and \eqref{Eq:cow}, we establish
\begin{Eq*}
I_{34}=&\partial_t\gamma^{rr}|\partial_r f+w f|^2-\kl(2\partial_r\gamma^{rr}w+\frac{(n-1)^2}{2}\gamma^{rr}\kr) f\partial_t f
+\partial_t\gamma^{rr}\kl(w'+w^2-\frac{(n-1)^2}{4}\kr) f^2\\
=&\partial_t\gamma^{rr}|\partial_r f+w f|^2-\kl(2\partial_r\gamma^{rr}w+\frac{(n-1)^2}{2}\gamma^{rr}\kr) f\partial_t f
+\frac{(n-1)(n-3)}{4\sinh^2r}\partial_t\gamma^{rr} f^2\\
|I_{34}|\lesssim&\varepsilon'(r)r^{-1}\kl(|\partial_t f|^2+|\partial_r f+w f|^2+\frac{| f|^2}{r^2}\kr)\lesssim\varepsilon'(r)r^{-1}(|\tpartial f|^2+r^{-2}| f|^2).
\end{Eq*}

Combining all these results, we arrive at
\begin{Eq*}
\sum_{k=1}^4\kl|\int_{S_0^T}I_{3k}\kr|
\lesssim& \sum_j\int_{[0,T]\times A_j}\varepsilon'(r)r^{-1}(|\tpartial f|^2+r^{-2}| f|^2)\\
\lesssim& \sum_{j}\varepsilon'_j\|f\|_{LE(S_0^T)}^2\lesssim \varepsilon'\|f\|_{LE(S_0^T)}^2.
\end{Eq*}
Substituting these into \eqref{Eq:GioJg} and reviewing the energy-estimate argument for the unperturbed case,
we finally arrive at
\begin{Eq*}
\|f\|_{E(S_0^T)}^2\lesssim \|\tpartial f(0)\|_{L^2(\rH^n)}^2+\int_{S_0^T}|\square_\rH f+\square_\gamma f|\cdot|\tpartial f|+\varepsilon'\|f\|_{LE(S_0^T)}^2.
\end{Eq*}

\subsection{Local energy estimate}
In this subsection, we set
$
(X,\chi,Y)=\kl(h_R\partial_r,h_Rw,Y^t\partial_t+Y^r\partial_r+Y^A\partial_A\kr)$,
with $h_R:=\frac{r}{R+r}$, and the components of $Y$ are given by
\begin{Eq*}
Y^t=&\frac{1}{2}\partial_r\gamma^{tr}h_Rw f^2,\\
Y^r=&-\frac{1}{2}\kl(\partial_t\gamma^{tr}h_R+\snabla_B\gamma^{rB}h_R+\gamma^{rr}h_R'\kr)w f^2-\frac{(n-1)^2}{8}\gamma^{rr}h_R f^2,\\
Y^A=&\frac{1}{2}\partial_r\gamma^{rA}h_Rw f^2.
\end{Eq*}
Through direct computation, we obtain
\begin{Eq*}
J_\gamma^t=&-h_RT_\gamma^{tr}+\frac{1}{2}\gamma^{tr}\partial_r\chi f^2-\gamma^{t\alpha}\chi f\partial_\alpha f+Y^t\\
=&-\gamma^{t\alpha}h_R\partial_\alpha f(\partial_r f+w f)+\frac{1}{2}\gamma^{tr}\partial_r(h_Rw) f^2+\frac{1}{2}\partial_r\gamma^{tr}h_Rw f^2\\
=&-\gamma^{tt}h_R\partial_t f(\partial_r f+w f)-\gamma^{tA}h_R(\partial_r f+w f)\partial_A f-\gamma^{tr}h_R\kl(\partial_r f+w f\kr)^2+\frac{1}{2}\hpartial_r\kl(\gamma^{tr}h_Rw f^2\kr).
\end{Eq*}
Noting that $|h_R|\leq 1$ and applying \Hs{Hs:hoP}, 
we conclude that
$|\int_{\rH^n}J_\gamma^t|\lesssim \varepsilon' \|\tpartial f\|_{L^2(\rH^n)}^2$.
Referring back to the procedure for the unperturbed case, 
this ensures that the boundary integral in \eqref{Eq:GioJg} can also be absorbed into the previously established energy.

On the other hand, to analyze the inner integral in \eqref{Eq:GioJg},
we first compute the nonzero components of $\nabla_\alpha X_\beta$ as
\begin{Eq*}
\nabla_r X_r=\partial_r X^r=h_R',\quad \nabla_A X_{B}=g_{BD}\Gamma^D{}_{Ar}X^r=\frac{h_R}{\tanh r} g_{AB},
\end{Eq*}
which leads to the expressions
\begin{Eq*}
T_{\gamma}^{rr}\nabla_r X_r=&h_R'\kl(\gamma^{r\sigma}\partial_\sigma f\partial_r f-\frac{1}{2}\gamma^{\sigma \delta}\partial_\sigma f\partial_\delta f\kr),\\
T_\gamma^{AB}\nabla_AX_B=&\frac{h_R}{\tanh r}\kl(\gamma^{A \sigma}\partial_\sigma f\partial_A f-\frac{n-1}{2}\gamma^{\sigma \delta}\partial_\sigma f\partial_\delta f \kr).
\end{Eq*}
As a result, \eqref{Eq:DoJg} simplifies to
\begin{Eq*}
\nabla_\alpha J_\gamma^\alpha=&-h_R\square_{\gamma} f( \partial_r f+w f)+\frac{1}{2}X^r\kl(\partial_r \gamma^{\sigma \delta}+2\Gamma^\sigma{}_{r\alpha} \gamma^{\alpha \delta}\kr)\partial_\sigma f\partial_\delta f-T_{\gamma}^{r r}\nabla_r X_r
-T_{\gamma}^{A B}\nabla_A X_B\\
&\quad
-\gamma^{\alpha \beta}(h_Rw)\partial_\alpha  f\partial_\beta f+\frac{1}{2}\nabla_\alpha \kl(\gamma^{\alpha \beta}\partial_\beta(h_Rw)\kr) f^2+\nabla_\alpha Y^\alpha\\
=&-h_R\square_{\gamma} f( \partial_r f+w f)
+\frac{\partial_r\gamma^{\sigma \delta}h_R}{2}\partial_\sigma f\partial_\delta f 
-h_R'\kl(\gamma^{r\sigma}\partial_\sigma f\partial_r f
-\frac{1}{2}\gamma^{\sigma \delta}\partial_\sigma f\partial_\delta f\kr)
\\
&\quad+\frac{1}{2}\partial_t\kl(\gamma^{tr}\partial_r(h_Rw)\kr)f^2
+\frac{1}{2}\hpartial_r\kl(\gamma^{rr}\partial_r(h_Rw)\kr)f^2
+\frac{1}{2}\snabla_A \kl(\gamma^{rA}\partial_r(h_Rw)\kr)f^2
+\partial_t\kl(\frac{1}{2}\partial_r\gamma^{tr}h_Rw f^2\kr)\\
&\quad+\hpartial_r\kl(-\frac{1}{2}\kl(\partial_t\gamma^{tr}h_R+\snabla_B\gamma^{rB}h_R+\gamma^{rr}h_R'\kr)w f^2-\frac{(n-1)^2}{8}\gamma^{rr}h_R f^2\kr)
+\snabla_A  \kl(\frac{1}{2}\partial_r\gamma^{rA}h_Rw f^2\kr)\\
=&\kl(\frac{1}{2}\kl(\partial_r\gamma^{tt}h_R+\gamma^{tt}h_R'\kr)|\partial_t f|^2
+\kl(\partial_r\gamma^{tA}h_R+\gamma^{tA}h_R'\kr)\partial_t f \partial_A  f
+\frac{1}{2}\kl(\partial_r\gamma^{AB}h_R+\gamma^{AB}h_R'\kr) \partial_A  f\partial_B f\kr)\\
&+\kl(\partial_r\gamma^{tr}h_R\partial_t f\partial_r f
+\frac{1}{2}\partial_t\kl(\gamma^{tr}\partial_r(h_Rw)\kr) f^2
+\partial_t\kl(\frac{1}{2}\partial_r\gamma^{tr}h_Rw f^2\kr)
+\hpartial_r\kl(-\frac{1}{2}\partial_t\gamma^{tr}h_Rw f^2\kr)\kr)\\
&+\kl(\partial_r\gamma^{rA}h_R\partial_r f \partial_A  f
+\frac{1}{2}\snabla_A \kl(\gamma^{rA}\partial_r(h_Rw)\kr) f^2
+\snabla_A \kl(\frac{1}{2}\partial_r\gamma^{rA}h_Rw f^2\kr)
+\hpartial_r\kl(-\frac{1}{2}\snabla_A \gamma^{rA}h_Rw f^2\kr)\kr)\\
&+\frac{1}{2}\kl(\kl(\partial_r\gamma^{rr}h_R-\gamma^{rr}h_R'\kr)|\partial_r f|^2
+\hpartial_r\kl(\gamma^{rr}\partial_r(h_Rw)\kr) f^2
+\hpartial_r\kl(-\gamma^{rr}h_R'w f^2\kr)
-\frac{(n-1)^2}{4}\hpartial_r(\gamma^{rr}h_R f^2)\kr)\\
&-h_R\square_{\gamma} f(\partial_r f+w f)\\
\equiv& I_{41}+I_{42}+I_{43}+\frac{1}{2}I_{44}-h_R\square_{\gamma} f(\partial_r f+w f).
\end{Eq*}
First, for $I_{41}$, 
invoking \Hs{Hs:hoP}, we immediately obtain $|I_{41}|\lesssim (\varepsilon'(r)r^{-1}+\varepsilon' h_R')|\tpartial f|^2$. 
Next, for $I_{42}$, 
applying \Hs{Hs:hoP}, we derive
\begin{Eq*}
I_{42}=&\partial_r\gamma^{tr}h_R\partial_t f(\partial_r f+w f)
-\partial_t\gamma^{tr}h_Rw f(\partial_r f+w f),\\
|I_{42}|\lesssim&\varepsilon'(r)r^{-1}\kl(|\partial_t f|^2+|\partial_r f+w f|^2+\frac{| f|^2}{r^{2}}\kr)\lesssim \varepsilon'(r)r^{-1}(|\tpartial f|^2+r^{-2}| f|^2).
\end{Eq*}
Similarly, for $I_{43}$, using \Hs{Hs:hoP}, we obtain
\begin{Eq*}
I_{43}=&\partial_r\gamma^{rA}h_R(\partial_r f+w f) \partial_A  f
-\snabla_A \gamma^{rA}h_Rw f(\partial_r f+w f),\\
|I_{43}|\lesssim& \varepsilon'(r)r^{-1}\kl(|\partial_r f+w f|^2+|\spartial f|^2+\frac{| f|^2}{r^2}\kr)\lesssim \varepsilon'(r)r^{-1}(|\tpartial f|^2+r^{-2}| f|^2).
\end{Eq*}
Finally, for $I_{44}$, leveraging both \eqref{Eq:cow} 
and the fact that the coefficients satisfy \Hs{Hs:hoP}, 
we establish
\begin{Eq*}
I_{44}
=&\kl(\partial_r\gamma^{rr}h_R-\gamma^{rr}h_R'\kr)|\partial_r f+w f|^2
-2\partial_r\gamma^{rr}h_Rw f(\partial_r f+w f)
-\frac{(n-1)^2}{4}\gamma^{rr}h_R\hpartial_r( f^2)\\
&+\kl(\partial_r\gamma^{rr}h_R+\gamma^{rr}h_R'\kr)\kl(w^2+w'-\frac{(n-1)^2}{4}\kr) f^2+\gamma^{rr}h_R\kl(w''+2ww'\kr) f^2\\
=&\kl(\partial_r\gamma^{rr}h_R-\gamma^{rr}h_R'\kr)|\partial_r f+w f|^2
-\kl(2\partial_r\gamma^{rr}h_Rw+\frac{(n-1)^2}{2}\gamma^{rr}h_R\kr) f(\partial_r f+w f)\\
&+\frac{(n-1)(n-3)}{4\sinh^2r}\kl(\partial_r\gamma^{rr}h_R+\gamma^{rr}h_R'\kr) f^2-\frac{(n-1)(n-3)}{2\sinh^2 r\tanh r}\gamma^{rr}h_R f^2,\\
|I_{44}|\lesssim&\varepsilon'(r)r^{-1}\kl(|\tpartial f|^2+r^{-2}| f|^2\kr)+\varepsilon' h_R'|\tpartial f|^2+(n-3)\frac{h_R}{\sinh^{2}r\tanh r}| f|^2
\end{Eq*}

Combining all these results, we obtain
\begin{Eq*}
\sum_{k=1}^4\kl|\int_{S_0^T}I_{4k}\kr|
\lesssim& \sum_j\int_{[0,T]\times A_j}\varepsilon'(r)r^{-1}\kl(|\tpartial f|^2+r^{-2}| f|^2\kr)\\
&\quad+\varepsilon' \int_{S_0^T}h_R'|\tpartial f|^2+(n-3)\frac{h_R}{\sinh^{2}r\tanh r}| f|^2\\
\lesssim &\varepsilon'\kl(\|f\|_{LE(S_0^T)}^2+\int_{S_0^T}h_R'|\tpartial f|^2+(n-3)\frac{h_R}{\sinh^{2}r\tanh r}| f|^2\kr).
\end{Eq*}
Substituting these into \eqref{Eq:GioJg},
incorporating the energy estimate for the perturbed case,
and reviewing the local-energy argument for the unperturbed case,
we finally derive
\begin{Eq*}
LHS:=&\| f\|_{E\cap LE(S_0^T)}^2+\int_{S_0^T}h_R'|\tpartial f|^2+(n-3)\frac{h_R}{\sinh^{2}r\tanh r}| f|^2\\
\lesssim&\|\tpartial f(0)\|_{L^2(\rH^n)}^2+\varepsilon' LHS+\int_{S_0^T}|\square_\rH f+\square_\gamma f|\cdot|\tpartial f|.
\end{Eq*}
This immediately establishes \Tm{Tm:T} for the case $P=\square_{\rH}+\nabla_\alpha(\gamma^{\alpha\beta}\nabla_\beta)$. 

\subsection{Treatment of lower-order perturbations}
Building on the result established in the previous subsection,
we obtain
\begin{Eq*}
\| f\|_{[E\cap LE](S_0^T)}
\lesssim&\|\tpartial f(0)\|_{L^2(\rH^n)}+\|\square_\rH f+\square_\gamma f\|_{[L^1L^2+LE^*](S_0^T)}\\
\lesssim&\|\tpartial f(0)\|_{L^2(\rH^n)}+\|Pf\|_{[L^1L^2+LE^*](S_0^T)}+\|(G^\beta\partial_\beta f,Vf)\|_{LE^*(S_0^T)}.
\end{Eq*}
Thus, to complete the proof of \Tm{Tm:T},
it remains to control the error term arising from lower-order perturbations.
Applying \Hs{Hs:hoP}, we derive
\begin{Eq*}
\|G^\beta\partial_\beta f\|_{LE^*(S_0^T)}
\lesssim&\sum_{j}2^{j/2}\kl\|\kl(G^t\partial_t f,G^r(\partial_r f+w f),G^rw f,G^{A}\partial_{A} f\kr)\kr\|_{L^2L^2([0,T]\times A_j)}\\
\lesssim&\sum_{j}2^{j/2}\varepsilon'_j\|r^{-1}(\tpartial f,r^{-1} f)\|_{L^2L^2([0,T]\times A_j)}\\
\lesssim&\sum_j\varepsilon'_j\|f\|_{LE(S_0^T)}\lesssim \varepsilon'\|f\|_{LE(S_0^T)},\\
\|V f\|_{LE^*(S_0^T)}
\lesssim&\sum_{j}2^{j/2}\varepsilon'_j\|r^{-2} f\|_{L^2L^2([0,T]\times A_j)}
\lesssim\varepsilon'\|f\|_{LE(S_0^T)}.
\end{Eq*}
This concludes the proof.

\section{Global existence for 
semilinear wave equations}\label{Sn:SLWgwp}

In this section, we prove \Tm{Tm:P2}. Throughout this section, we only consider radial solutions and radial coefficients. In particular, all angular derivatives vanish on the iterates, and any $L^q(\rS^{n-1})$ norm reduces to an absolute value up to a harmless constant.

\subsection{Iteration scheme and bootstrap quantities}

We begin by introducing the iteration scheme. Set $\phi_0=0$, and define $\phi_{k+1}$ by
\begin{Eq*}
\begin{cases}
P\phi_{k+1}=F_{p_1}(\phi_k)+F_{p_2}(\partial_t\phi_k)+F_{p_3}(\partial_r\phi_k),\\
\phi_{k+1}(0,r)=\varepsilon\psi_0(r),\quad \partial_t\phi_{k+1}(0,r)=\varepsilon\psi_1(r).
\end{cases}
\end{Eq*}
Here $p_i>1$,  $|F_{p_i}(s)|+|sF_{p_i}'(s)|\lesssim |s|^{p_i}$ for $i=1,2,3$, and $\varepsilon>0$ is sufficiently small.

Since $P$ satisfies \eqref{Eq:hoP2}, it also satisfies \Hs{Hs:hoP}. Moreover, without loss of generality, we may first truncate the initial data and then pass to the limit by approximation. Therefore, \Tm{Tm:T} is applicable to each iterate.

For the vector fields $Z$ introduced in \Sn{Sn:p}, we define
\begin{Eq*}
\varE_k:=\|\phi_k\|_{\mathcal{E}_1},\qquad
\varF_k:=\|\phi_{k+1}-\phi_k\|_{\mathcal{E}_0}.
\end{Eq*}
where $\|f\|_{\mathcal{E}_m}:=\|Z^{\leq m}f\|_{[E\cap LE](S_0^\infty)}$ for $m=0,1$.

Our goal is twofold. First, we prove that the iterates are uniformly bounded in $\mathcal{E}_1$, namely, $\varE_k\lesssim \varepsilon$ uniformly in $k$. Next, we show that the iteration is contractive with respect to the base norm $\mathcal{E}_0$ by establishing
$\varF_{k+1}\leq \frac12 \varF_k$.
These two estimates together imply that the iteration converges in $\mathcal{E}_0$ to a global finite-energy solution $\phi$.

Before proceeding, we record the following bound on the initial data:
\begin{Eq}\label{Eq:coi2}
\|\tpartial^{\leq 1}Z^{\leq 1}\phi_k(0)\|_{L^2}\lesssim \varepsilon,
\end{Eq}
where the implicit constant is independent of $k$. The proof of \eqref{Eq:coi2} will be given in \Sn{Sn:pocoi2}.

\subsection{A priori bounds for the iterates}

By \eqref{Eq:hti}, together with the argument used in \Sn{Sn:mii}, we obtain, for any $0<\delta_0,\delta_1\ll 1$,
\begin{Eq*}
&\kl\|r^{-\frac{1}{2}}\sinh^{\frac{n-1}{2}}r\, Z^{\leq 1}\phi_k\kr\|_{L^\infty L^\infty}
+\kl\|r^{-\frac{1}{2}+\delta_0}\kl<r\kr>^{-\delta_0-\delta_1}(\tpartial,r^{-1})Z^{\leq 1}\phi_k\kr\|_{L^2L^2}\\
\lesssim& \|\phi_k\|_{\mathcal{E}_1}
=\varE_k.
\end{Eq*}

We now prove the uniform bound for $\varE_k$ by induction. Clearly, $\varE_0=0$. Assuming that $\varE_k\lesssim \varepsilon$ for some $k\geq 0$, 
we shall show that the same bound also holds for $\varE_{k+1}$.

By \eqref{Eq:Pgei}, \eqref{Eq:coi2}, and the dual form of \eqref{Eq:sf}, we obtain
\begin{Eq*}
\varE_{k+1}\lesssim&
\|\tpartial Z^{\leq 1}\phi_{k+1}(0)\|_{L^2}
+\kl\|Z^{\leq 1}F_{p_1}(\phi_k)\kr\|_{LE^*}
+\kl\|Z^{\leq 1}F_{p_2}(\partial_t\phi_k)\kr\|_{LE^*}\\
&+\kl\|Z^{\leq 1}F_{p_3}(\partial_r\phi_k)\kr\|_{LE^*}
+\|[Z,P]\phi_{k+1}\|_{LE^*}\\
\lesssim& \varepsilon + I_{51}+I_{52}+I_{53}+I_{54},
\end{Eq*}
where
\begin{Eq*}
I_{51}:=&\kl\|r^{\frac{1}{2}-\delta_0}\kl<r\kr>^{\delta_0+\delta_1}Z^{\leq 1}F_{p_1}(\phi_k)\kr\|_{L^2L^2},\quad
&I_{52}:=&\kl\|r^{\frac{1}{2}-\delta_0}\kl<r\kr>^{\delta_0+\delta_1}Z^{\leq 1}F_{p_2}(\partial_t\phi_k)\kr\|_{L^2L^2},\\
I_{53}:=&\kl\|r^{\frac{1}{2}-\delta_0}\kl<r\kr>^{\delta_0+\delta_1}Z^{\leq 1}F_{p_3}(\partial_r\phi_k)\kr\|_{L^2L^2},\quad
&I_{54}:=&\|[Z,P]\phi_{k+1}\|_{LE^*}.
\end{Eq*}

We first estimate $I_{51}$. Since $p_1<\frac{n+2}{n-2}$, we have
$\frac{p_1+3}{2}+\frac{(p_1-1)(1-n)}{2}>0$.
Hence,
\begin{Eq*}
I_{51}\lesssim&
\kl\|r^{-\frac{p_1+2}{2}+\delta_0}\kl<r\kr>^{-\delta_0-\delta_1}
\sinh^{\frac{(p_1-1)(n-1)}{2}}r\, Z^{\leq 1}F_{p_1}(\phi_k)\kr\|_{L^2L^2}\\
&\qquad\cdot
\kl\|r^{\frac{p_1+3}{2}-2\delta_0}\kl<r\kr>^{2\delta_0+2\delta_1}
\sinh^{\frac{(p_1-1)(1-n)}{2}}r\kr\|_{L_r^\infty}\\
\lesssim&
\kl\|r^{-\frac{1}{2}}\sinh^{\frac{n-1}{2}}r\, \phi_k\kr\|_{L^\infty L^\infty}^{p_1-1}
\kl\|r^{-\frac{3}{2}+\delta_0}\kl<r\kr>^{-\delta_0-\delta_1}Z^{\leq 1}\phi_k\kr\|_{L^2L^2}\\
\lesssim& \varE_k^{p_1}
\lesssim \varepsilon^{p_1-1}\varE_k.
\end{Eq*}

Next we estimate $I_{52}$. Since $p_2<\frac{n}{n-2}$, we have
$\frac{p_2+1}{2}+\frac{(p_2-1)(1-n)}{2}>0$.
Therefore,
\begin{Eq*}
I_{52}\lesssim&
\kl\|r^{-\frac{p_2}{2}+\delta_0}\kl<r\kr>^{-\delta_0-\delta_1}
\sinh^{\frac{(p_2-1)(n-1)}{2}}r\, Z^{\leq 1}F_{p_2}(\partial_t\phi_k)\kr\|_{L^2L^2}\\
&\qquad\cdot
\kl\|r^{\frac{p_2+1}{2}-2\delta_0}\kl<r\kr>^{2\delta_0+2\delta_1}
\sinh^{\frac{(p_2-1)(1-n)}{2}}r\kr\|_{L_r^\infty}\\
\lesssim&
\kl\|r^{-\frac{1}{2}}\sinh^{\frac{n-1}{2}}r\, \partial_t\phi_k\kr\|_{L^\infty L^\infty}^{p_2-1}
\kl\|r^{-\frac{1}{2}+\delta_0}\kl<r\kr>^{-\delta_0-\delta_1}Z^{\leq 1}\partial_t\phi_k\kr\|_{L^2L^2}\\
\lesssim&
\kl\|r^{-\frac{1}{2}}\sinh^{\frac{n-1}{2}}r\, Z\phi_k\kr\|_{L^\infty L^\infty}^{p_2-1}
\kl\|r^{-\frac{1}{2}+\delta_0}\kl<r\kr>^{-\delta_0-\delta_1}\tpartial Z^{\leq 1}\phi_k\kr\|_{L^2L^2}\\
\lesssim& \varE_k^{p_2}
\lesssim \varepsilon^{p_2-1}\varE_k.
\end{Eq*}

Before estimating $I_{53}$, we recall that for a radial function $f(t,r)$, one may identify $|\partial_r f|$ with $|Kf|$, and hence identify $|\partial_{t,r}f|$ with $|Zf|$, up to harmless constants.  Consequently,
\begin{Eq*}
|Z^{\leq 1}\partial_r\phi_k|
\lesssim
|\partial_r Z^{\leq 1}\phi_k|
\lesssim
|\kl<r\kr>(\tpartial,r^{-1})Z^{\leq 1}\phi_k|.
\end{Eq*}

We now turn to $I_{53}$. Since $p_3<\frac{n}{n-2}$, we similarly have
\begin{Eq*}
I_{53}\lesssim&
\kl\|r^{-\frac{p_3}{2}+\delta_0}\kl<r\kr>^{-\delta_0-\delta_1-1}
\sinh^{\frac{(p_3-1)(n-1)}{2}}r\, Z^{\leq 1}F_{p_3}(\partial_r\phi_k)\kr\|_{L^2L^2}\\
&\qquad\cdot
\kl\|r^{\frac{p_3+1}{2}-2\delta_0}\kl<r\kr>^{2\delta_0+2\delta_1+1}
\sinh^{\frac{(p_3-1)(1-n)}{2}}r\kr\|_{L_r^\infty}\\
\lesssim&
\kl\|r^{-\frac{1}{2}}\sinh^{\frac{n-1}{2}}r\, \partial_r\phi_k\kr\|_{L^\infty L^\infty}^{p_3-1}
\kl\|r^{-\frac{1}{2}+\delta_0}\kl<r\kr>^{-\delta_0-\delta_1-1}Z^{\leq 1}\partial_r\phi_k\kr\|_{L^2L^2}\\
\lesssim&
\kl\|r^{-\frac{1}{2}}\sinh^{\frac{n-1}{2}}r\, Z\phi_k\kr\|_{L^\infty L^\infty}^{p_3-1}
\kl\|r^{-\frac{1}{2}+\delta_0}\kl<r\kr>^{-\delta_0-\delta_1}(\tpartial,r^{-1})Z^{\leq 1}\phi_k\kr\|_{L^2L^2}\\
\lesssim& \varE_k^{p_3}
\lesssim \varepsilon^{p_3-1}\varE_k.
\end{Eq*}

\subsection{Completion of the iteration argument}
We next estimate the commutator term $I_{54}$. We first restrict ourselves to the region $\{r>1\}$. By \eqref{Eq:hoP2}, for any radial function $f(t,r)$ we have
\begin{Eq*}
\kl|[Z,P]f\kr|
=&\kl|[Z,\partial_\alpha(\gamma^{\alpha\beta}\partial_\beta)+2w\gamma^{r\beta}\partial_\beta+G^\alpha\partial_\alpha+V]f\kr|\\
\lesssim&
\kl|Z\partial_\alpha\gamma^{\alpha\beta}\cdot\partial_\beta f\kr|
+\kl|Z\gamma^{\alpha\beta}\cdot\partial_\alpha\partial_\beta f\kr|\\
&+|Z(2w\gamma^{r\beta}+G^\beta)\cdot\partial_\beta f|
+|ZV\cdot f|\\
\lesssim&
\kl|\kl(\partial_{t,r}^{\leq 2}\gamma,\partial_{t,r}^{\leq 1}G,\partial_{t,r}V\kr)\kr|
\cdot
\kl|\partial_{t,r}^{\leq 1}Z^{\leq 1}f\kr|\\
\lesssim& \varepsilon'(r)r^{-2}|\tpartial^{\leq 1}Z^{\leq 1}f|.
\end{Eq*}
It follows that
\begin{Eq*}
I_{54(r>1)}
\lesssim&
\kl\|\varepsilon'(r)r^{-2}\tpartial^{\leq 1}Z^{\leq 1}\phi_{k+1}\kr\|_{LE^*(r>1)}\\
\lesssim&
\sum_{j>0}\varepsilon'_j
\kl\|\kl(r^{-\frac{1}{2}}\tpartial Z^{\leq 1}\phi_{k+1},\, r^{-\frac{3}{2}}Z^{\leq 1}\phi_{k+1}\kr)\kr\|_{L^2L^2(\rR\times A_j)}\\
\lesssim&
\|Z^{\leq 1}\phi_{k+1}\|_{LE}\sum_{j>0}\varepsilon'_j
\lesssim \varepsilon'\varE_{k+1}.
\end{Eq*}

For the region $\{r<1\}$, the local regularity of $P$ near the origin and the smallness of the coefficients imply that $I_{54(r<1)}$ satisfies the same bound. Therefore,
$I_{54}\lesssim \varepsilon'\varE_{k+1}$.

Combining the above bounds for $I_{51}$, $I_{52}$, $I_{53}$, and $I_{54}$, we obtain
\begin{Eq*}
\varE_{k+1}\leq C\kl(\varepsilon+\varepsilon^{\min\{p_1,p_2,p_3\}-1}\varE_k+\varepsilon'\varE_{k+1}\kr).
\end{Eq*}

Hence, if
$
\varE_k\leq 2C\varepsilon$
and
$
\max\{\varepsilon^{\min\{p_1,p_2,p_3\}-1},\,\varepsilon'\}\leq (4C)^{-1}$,
then
$
\varE_{k+1}\leq 2C\varepsilon$.
By induction, it follows that 
$\varE_k\lesssim \varepsilon$ uniformly in $k$.

Finally, by repeating the argument in the previous two subsections, and arguing exactly as in \Sn{Sn:car}, we obtain $
\varF_{k+1}\leq \frac12 \varF_k$.
We omit the details, since the proof is entirely analogous.

Therefore, the sequence $\{\phi_k\}_{k\geq 0}$ converges to a global radial finite-energy solution $\phi$ of \eqref{Eq:P2}.
Moreover, by the uniform bound for $\varE_k$ and the radial equivalence between $Z$-derivatives and $(\partial_t,\partial_r)$-derivatives, the limit $\phi$ satisfies
\begin{Eq*}
\|r^{-\frac{1}{2}}(\sinh r)^{\frac{n-1}{2}}\partial_{t,r}^{\leq 1}\phi\|_{L^\infty L^\infty(S_0^\infty)}
+\|\partial_{t,r}^{\leq 1}\phi\|_{[E\cap LE](S_0^\infty)}
\lesssim \varepsilon.
\end{Eq*}
This completes the proof of \Tm{Tm:P2}.

\subsection{Proof of Claim~(\ref{Eq:coi2})}\label{Sn:pocoi2}
In this subsection, we prove Claim \eqref{Eq:coi2}. 
Without loss of generality, we omit the subscript $k$ throughout this subsection.

For $n\geq 3$, we first note that
\begin{Eq*}
\partial_r(\sinh^{n-2}r\cosh r)
=(n-2)\sinh^{n-3}r\cosh^2 r+\sinh^{n-1}r
\approx \sinh^{n-3}r\cosh^2 r.
\end{Eq*}
Using \eqref{Eq:ki2}, for any $f(r)$ we then obtain
\begin{Eq}\label{Eq:hti2}
\|wf\|_{L^2}^2
\lesssim&
\int_0^\infty f^2\,\partial_r(\sinh^{n-2}r\cosh r)\d r\\
\lesssim&
\int_0^\infty (\partial_r f)^2
\frac{\sinh^{2n-4}r\cosh^2 r}
{\partial_r(\sinh^{n-2}r\cosh r)}
\d r\\
\approx&
\int_0^\infty (\partial_r f)^2\sinh^{n-1}r\d r
=\|\partial_r f\|_{L^2}^2,
\end{Eq}
and thus $\|\tpartial_rf\|_{L^2}\lesssim\|\partial_r f\|_{L^2}$. Therefore, we readily obtain 
\begin{Eq*}
\|\tpartial Z^{\leq 1}\phi(0)\|_{L^2}
\lesssim
\|\partial_t^2\phi(0)\|_{L^2}
+\varepsilon\kl(
\|\partial_r^{\leq 1}\psi_1\|_{L^2}
+\|\partial_r^{\leq 2}\psi_0\|_{L^2}
\kr).
\end{Eq*}

Meanwhile, by the equation \eqref{Eq:P2}, we have
\begin{Eq*}
\partial_t^2\phi(0)
=&-(1+\gamma^{tt})^{-1}_{t=0}\kl(I_{61}+I_{62}+I_{63}+I_{64}\kr).
\end{Eq*}
where
\begin{Eq*}
I_{61}=&F_{p_1}(\varepsilon\psi_0),\qquad
I_{62}=F_{p_2}(\varepsilon\psi_1),\qquad
I_{63}=F_{p_3}(\varepsilon\partial_r\psi_0),\\
I_{64}=&-\kl((1+\gamma^{tt})\partial_t^2+P\kr)\phi(0).
\end{Eq*}
Note that $I_{64}$ contains no second-order time derivative of $\phi$, and hence depends only on the initial data $(\psi_0,\psi_1)$.

We temporarily postpone the estimate of $I_{61}$ and begin with $I_{62}$ instead. By \eqref{Eq:hti2} and \eqref{Eq:coid2}, and using the fact that $1<p_2<\frac{n}{n-2}$, we have
\begin{Eq*}
\|I_{62}\|_{L^2}
\lesssim&
\varepsilon^{p_2}\kl\||\psi_1|^{p_2}\kr\|_{L^2}\\
\lesssim&
\varepsilon^{p_2}\kl\|r^{-1}\psi_1\kr\|_{L^2}
\kl\|r^{-\frac{1}{2}}\sinh^{\frac{n-1}{2}}r\,\psi_1\kr\|_{L^\infty}^{p_2-1}\kl\|r^{\frac{p_2+1}{2}}\sinh^{-\frac{n-1}{2}(p_2-1)}r\kr\|_{L_r^\infty}\\
\lesssim&\varepsilon^{p_2}\cdot \|\tpartial_r\psi_1\|_{L^2}^{p_2}\cdot 1\lesssim \varepsilon^{p_2}\cdot \|\partial_r^{\leq 1}\psi_1\|_{L^2}^{p_2}\lesssim \varepsilon^{p_2}.
\end{Eq*}
Similarly, applying the same argument, we obtain $\|I_{63}\|_{L^2}\lesssim \varepsilon^{p_3}$.

For $I_{61}$, the main difference is that here we allow
\begin{Eq*}
\frac{n}{n-2}<p_1<\frac{n+2}{n-2}.
\end{Eq*}
Therefore, in order to avoid the singular behavior near $r=0$,
we use \eqref{Eq:ki2} and \eqref{Eq:hti2} again to obtain
\begin{Eq*}
\sup_r\kl|\psi_0^2(r)(1+r)\kr|
\lesssim&\int_0^{+\infty}|\partial_r \psi_0|^2(1+r)^2\d r\\
\lesssim&
\int_0^\infty |\partial_r^{\leq 1} \psi_0|^2
\sinh^{n-3}r\cosh^2r
\d r\\
\approx&\|w\partial_r^{\leq 1} \psi_0\|_{L^2}^2
\lesssim\|\partial_r^{\leq 2} \psi_0\|_{L^2}^2,
\end{Eq*}
Then, $\|\psi_0\|_{L^\infty}\lesssim \|\partial_r^{\leq 2}\psi_0\|_{L^2}\lesssim 1$ and thus
\begin{Eq*}
\|I_{61}\|_{L^2}
\lesssim
\varepsilon^{p_1}\kl\||\psi_0|^{p_1}\kr\|_{L^2}
\lesssim
\varepsilon^{p_1}\kl\|\psi_0\kr\|_{L^2}
\kl\|\psi_0\kr\|_{L^\infty}^{p_1-1}
\lesssim\varepsilon^{p_1}
\end{Eq*}

It remains to estimate $I_{64}$. Following \eqref{Eq:hoP2}, since the coefficients of $(\partial_r\psi_1,\partial_r^2\psi_0)$ are uniformly bounded, while the coefficients of $(\psi_1,\partial_r^{\leq 1}\psi_0)$ are all controlled by $w$, we obtain
\begin{Eq*}
\|I_{64}\|_{L^2}
\lesssim
\varepsilon\kl(
\|(\partial_r\psi_1,\partial_r^2\psi_0)\|_{L^2}
+\|w(\psi_1,\partial_r^{\leq 1}\psi_0)\|_{L^2}
\kr)
\lesssim
\varepsilon.
\end{Eq*}

Finally, since $|\gamma^{tt}|\leq \varepsilon'\ll 1$, we infer that
\begin{Eq*}
\|\tpartial Z^{\leq 1}\phi(0)\|_{L^2}
\lesssim
\varepsilon+\varepsilon^{p_1}+\varepsilon^{p_2}+\varepsilon^{p_3}
\lesssim
\varepsilon,
\end{Eq*}
which proves \eqref{Eq:coi2}.

\section{Appendix: local energy estimates implied by the resolvent estimates}\label{Sn:A}

In this appendix, we explain how a sharp high-frequency weighted resolvent
 estimate together with a low-frequency weighted limiting absorption
principle implies a spectral local energy estimate for the wave equation.
This
also clarifies the relation between the resolvent approach and the
physical-space multiplier estimates used in the main part of the paper.

We begin by setting up the relevant resolvent inputs.  Low-frequency resolvent estimates on asymptotically hyperbolic manifolds go back to
Mazzeo--Melrose \cite{MazzeoMelrose1987}, Mazzeo \cite{mazzeo}, and
Guillarmou \cite{Guillarmou2005}.  High-frequency resolvent estimates have been
studied by Vasy \cite{MR3117526}, Melrose--S\'a Barreto--Vasy
\cite{MR3169792}, Chen--Hassell \cite{MR3473907,MR3805767}, and Wang
\cite{WangSharpAHM}, among others.  These estimates are closely related to
local energy and smoothing estimates; see,
for example, Burq \cite{Burq} and D'Ancona \cite{MR3314497}.

The result below should be understood in the standard resolvent/Kato
smoothing framework.  In contrast to the estimates proved in the main body of
the paper, this method is naturally adapted to a time-independent spatial
operator and to exponentially decaying weights at infinity.  It is therefore
not a substitute for the physical-space multiplier estimates used in the main
part of the paper.

More precisely, let $(X^\circ,g)$ be an $n$-dimensional asymptotically
hyperbolic manifold in the sense of Mazzeo--Melrose, with compactification
denoted by $X$.  We write $L^2(X)$ for $L^2(X^\circ,\d V_g)$.  Let $\lambda$
be a boundary defining function for $X$, chosen as in the discussion of cAHM
in Subsection \ref{sec:cahm}.  If $\lambda$ is geodesic near infinity, then
$\lambda\sim e^{-r}$.  In this appendix only, let $\Delta_g$ denote the
non-negative Laplace--Beltrami operator and set
\begin{Eq*}
\rho:=\frac{n-1}{2},\qquad L:=\Delta_g-\rho^2.
\end{Eq*}
We use the self-adjoint realization of $L$ on $L^2(X)$.
Under the standard asymptotically hyperbolic spectral theory, the essential
spectrum of $L$ is $[0,\infty)$.  The operator $L$ may a priori have discrete
spectrum below $0$; this possibility will be excluded by the spectral
assumption below.  Thus, if $\Delta_{\mathbb H}$ in the main text denotes the
non-positive Laplacian, then the present non-negative $\Delta_g$ equals
$-\Delta_{\mathbb H}$ on the model space, and $L$ corresponds to
$-\Delta_{\mathbb H}-\rho^2$.  We write
\begin{Eq*}
R(z):=\left(L-z^2\right)^{-1}
\end{Eq*}
for the shifted resolvent.  For complex $z$, this notation is understood
through the meromorphic continuation of the asymptotically hyperbolic
resolvent; the threshold point corresponds to $z=0$.  In the estimates below
we take $z=\mu\pm i\varepsilon$, where $\mu\in\R$ and
$0<\varepsilon\leq\varepsilon_0$.  Throughout this appendix, $\lambda$ denotes
the boundary defining function; the spectral parameters are denoted by $\mu$,
$z$, and, when referring to Guillarmou's notation, $\zeta$.

We shall use the following fixed-strip consequence of
\cite[Theorem 1.1]{WangSharpAHM}, rewritten in the shifted notation above.
\begin{lemma}\label{La:Wang-res}
Let $n\ge 3$.
Assume that $(X^\circ,g)$ 
is a Cartan-Hadamard manifold, i.e., a complete simply-connected
Riemannian manifold with non-positive sectional curvature.
  Then, for every $a,b>0$ and every fixed
$0<\varepsilon_0<\min(a,b)$, there exist constants $R$ and $C$, depending on
$a,b$, $\varepsilon_0$, and the geometry, such that, for real $|\mu|\geq R$ and
$0<\varepsilon\leq\varepsilon_0$,
\begin{Eq}\label{Eq:Wang-res}
\|\lambda^aR(\mu\pm i\varepsilon)\lambda^b f\|_{L^2(X)}
\leq
C|\mu\pm i\varepsilon|^{-1}\|f\|_{L^2(X)}.
\end{Eq}
\end{lemma}

In the sequel we only use this estimate in a fixed strip near the real axis.
The restriction $\varepsilon_0<\min(a,b)$ is the present notation for
the weight condition in \cite[Theorem 1.1]{WangSharpAHM}, after restricting
the spectral parameter to the strip $|\operatorname{Im}z|\leq\varepsilon_0$.
When applying the lemma below with $b=a$, we choose the strip width smaller
than $a$.

The preceding lemma gives the high-frequency input.  To obtain a full spacetime
 estimate, one must also control the low-frequency spectral region.
We next recall a weighted limiting absorption principle in the low-frequency regime, 
obtained from the meromorphic resolvent theory of Mazzeo--Melrose and Guillarmou.
We use the following form of
\cite[Theorem 1.1 and Proposition 1.3]{Guillarmou2005}, rewritten in the
present shifted parameter.
\begin{lemma}\label{La:compact-LAP}
Assume that $(X^\circ,g)$ is a conformally compact asymptotically hyperbolic
manifold as in Subsection \ref{sec:cahm}, and that the metric is even to
infinite order in the sense of \cite[Definition 1.2]{Guillarmou2005}.  Assume
moreover that the meromorphically continued resolvent has no pole at the
threshold point; in the present variable this is $z=0$, equivalently
$\zeta=\rho$ in Guillarmou's notation.  Then, for every $M>0$ and every
$a>0$, there exist constants $\varepsilon_0>0$ and $C$ such that
\begin{Eq}\label{Eq:appendix-compact-lap}
\|\lambda^a R(\mu\pm i\varepsilon)\lambda^a f\|_{L^2(X)}
\leq
C\|f\|_{L^2(X)}
\end{Eq}
for $0\leq \mu\leq M$ and $0<\varepsilon\leq\varepsilon_0$.  Here
$R(\mu\pm i\varepsilon)$ denotes $(L-(\mu\pm i\varepsilon)^2)^{-1}$ in the
meromorphically continued sense.
\end{lemma}

\begin{proof}
Guillarmou \cite{Guillarmou2005} writes the modified resolvent as
\begin{Eq*}
(\Delta_g-\zeta(n-1-\zeta))^{-1}.
\end{Eq*}
With our interior dimension denoted by $n$, the change of variables
$\zeta=\rho+iz$, $\rho=(n-1)/2$, gives
\begin{Eq*}
\Delta_g-\zeta(n-1-\zeta)=L-z^2.
\end{Eq*}
Thus Guillarmou's modified resolvent is exactly the present convention
$R(z)=(L-z^2)^{-1}$.  In references where the boundary dimension is denoted
by $n$, the same formula is written with $\zeta(n-\zeta)$.  By
\cite[Proposition 1.3 and Theorem 1.4]{Guillarmou2005}, the evenness
assumption gives, for every $a>0$, a finite-meromorphic continuation of the
modified resolvent with the weighted mapping property
\begin{Eq*}
\lambda^a(L-z^2)^{-1}\lambda^a\in \mathcal B(L^2(X),L^2(X))
\end{Eq*}
away from possible finite-rank poles.  This is the point at which the evenness
hypothesis is used.  Guillarmou \cite[p.~14]{Guillarmou2005}, referring to
Patterson--Perry \cite[Lemma~4.9]{PattersonPerry2001},
recalls that this argument, together with Mazzeo's absence of embedded eigenvalues \cite{mazzeo}, gives no
pole on the line $\operatorname{Re}\zeta=\rho$ except possibly at the
threshold point $\zeta=\rho$.  The remaining threshold point is excluded by the
no-threshold-pole assumption, since $z=0$ corresponds to $\zeta=\rho$.
Hence the weighted resolvent is holomorphic near every point of the compact
segment $\{\zeta=\rho+i\mu:0\leq\mu\leq M\}$.  Since the poles are discrete,
this compact segment has a pole-free neighborhood.  Choosing
$\varepsilon_0>0$ sufficiently small so that
$\zeta=\rho+i(\mu\pm i\varepsilon)$ remains in this neighborhood for
$0\leq\mu\leq M$ and $0\leq\varepsilon\leq\varepsilon_0$, local boundedness and a
finite covering give
\eqref{Eq:appendix-compact-lap}.
\end{proof}

Combining the high-frequency input in \La{La:Wang-res} with the
low-frequency input in \La{La:compact-LAP}, we obtain the following weighted
half-wave resolvent estimate.  Its proof is postponed until after the proof of
\Tm{Tm:appendix-kato-le}.
\begin{lemma}\label{La:half-wave-res}
Assume that the assumptions of \La{La:Wang-res} and \La{La:compact-LAP} hold,
and assume in addition that $L$ has no point spectrum.  Since
$\sigma_{\mathrm{ess}}(L)=[0,\infty)$, the
self-adjoint operator $L$ is non-negative.  Let $H:=L^{1/2}$ be its
non-negative square root.  Then, for every $a>0$,
\begin{Eq}\label{Eq:appendix-half-wave-res}
\sup_{z\in\C\setminus\R}
\left\|
\lambda^a(H-z)^{-1}\lambda^a
\right\|_{L^2(X)\to L^2(X)}
<\infty.
\end{Eq}
\end{lemma}

For the reader's convenience, we collect the assumptions used in the final
estimate.
\begin{hypothesis}\label{Hs:appendix-resolvent}
We assume that $n\geq 3$ and that $(X^\circ,g)$ is a conformally compact
asymptotically hyperbolic manifold as in Subsection \ref{sec:cahm}, even to
infinite order in the sense
of \cite[Definition 1.2]{Guillarmou2005}, and also satisfies the high-frequency
geometric hypotheses of \La{La:Wang-res}, namely being simply connected and
having non-positive sectional curvatur(Cartan-Hadamard manifold).  Assume moreover that the
meromorphically continued resolvent has no pole at the threshold point $z=0$
and that $L$ has no point spectrum.
\end{hypothesis}

Under \Hs{Hs:appendix-resolvent}, the weighted high-frequency resolvent
estimate in \La{La:Wang-res}, the low-frequency limiting absorption
principle in \La{La:compact-LAP}, and therefore the weighted half-wave resolvent
estimate in \La{La:half-wave-res} hold.

\begin{theorem}[Spectral local energy estimate]
\label{Tm:appendix-kato-le}
Assume \Hs{Hs:appendix-resolvent}.  Let $I=[0,T)$ with
$0<T\leq\infty$, and fix $a>0$.
Let $\psi_0\in D(L^{1/2})$, $\psi_1\in L^2(X)$, and assume that
$\lambda^{-a}F\in L_t^2(I;L^2(X))$.  Let $\phi$ be the solution of
\begin{Eq}\label{Eq:appendix-wave}
\begin{cases}
(\partial_t^2+L)\phi=F,\\
\phi(0)=\psi_0,\qquad \partial_t\phi(0)=\psi_1.
\end{cases}
\end{Eq}
Then
\begin{Eq}\label{Eq:appendix-le}
&\|\lambda^a\partial_t\phi\|_{L_t^2(I;L^2(X))}
+
\|\lambda^aL^{1/2}\phi\|_{L_t^2(I;L^2(X))} \\
&\qquad\lesssim
\|L^{1/2}\psi_0\|_{L^2}
+
\|\psi_1\|_{L^2}
+
\|\lambda^{-a}F\|_{L_t^2(I;L^2(X))}.
\end{Eq}
The implicit constant may depend on $a$ and on the resolvent constants and
geometry, but is independent of the length of $I$, $\psi_0$, $\psi_1$, and
$F$.
\end{theorem}

\begin{remark}\label{Rk:appendix-limitation}
The estimate \eqref{Eq:appendix-le} is a spectral weighted spacetime energy
estimate, since it involves the nonlocal operator $L^{1/2}$.  On the model
hyperbolic space it should be compared with the usual physical-space local
energy norm through the shifted-gradient structure; on a general
asymptotically hyperbolic manifold it is best viewed as a spectral version of
local energy decay.

For a geodesic choice of $\lambda$ near infinity, the relation
$\lambda\sim e^{-r}$ shows that the
weight in \eqref{Eq:appendix-le} is exponential:
$\lambda^a\sim e^{-ar}$.  The forcing term is correspondingly measured in the
dual exponentially weighted norm
$\|\lambda^{-a}F\|_{L_t^2(I;L^2(X))}$, with
$\lambda^{-a}\sim e^{ar}$ near infinity.  Moreover, the whole argument is
naturally adapted to a fixed time-independent spatial operator, or equivalently
to a static product spacetime $\R_t\times X$.  Thus, although the resolvent
method gives a clean comparison result under the assumptions above, it does not
in its present form cover the class of quasilinear problems treated in the main
part of this paper, where the perturbations may be time-dependent and only
polynomial radial decay is assumed.
\end{remark}

\begin{proof}[Proof of \Tm{Tm:appendix-kato-le}]
Write $H:=L^{1/2}$ by functional calculus.  Under
\Hs{Hs:appendix-resolvent}, \La{La:half-wave-res} gives the uniform resolvent
estimate \eqref{Eq:appendix-half-wave-res}.  This is precisely Kato's
supersmoothness condition for $A=\lambda^a$ with respect to $H$; see
\cite{Kato1966,KatoYajima1989,ReedSimon3}.  Hence Kato's supersmoothing theorem gives
\begin{Eq}\label{Eq:appendix-hom-smoothing}
\|\lambda^a e^{itH}f\|_{L_t^2(I;L^2(X))}
\lesssim
\|f\|_{L^2}.
\end{Eq}
This estimate is obtained first on the full time line and then restricted to
the interval $I$.
We shall also use the standard retarded form of the inhomogeneous
supersmoothing estimate.  This is part of Kato's supersmoothing theorem: if
$A$ is $H$-supersmooth, then
\begin{Eq*}
\left\|
A\int_0^t e^{i(t-s)H}A^*G(s)\,\mathrm{d}s
\right\|_{L_t^2(I;L^2(X))}
\lesssim
\|G\|_{L_t^2(I;L^2(X))}.
\end{Eq*}
This estimate is first obtained on the full time line; the stated estimate on
$I$ follows by extending $G$ by zero and restricting the output to $I$.
Put $G=\lambda^{-a}F\in L_t^2(I;L^2(X))$.  Then $F=\lambda^aG$, and since
$A=\lambda^a$ is self-adjoint, $A^*G=F$.  Extending the forcing term by zero
outside $I$ gives, first for smooth compactly supported $F$ and then by density,
\begin{Eq}\label{Eq:appendix-inhom-smoothing}
&\left\|
\lambda^a\int_0^t e^{i(t-s)H}F(s)\,\mathrm{d}s
\right\|_{L_t^2(I;L^2(X))} \\
&\qquad\lesssim
\|\lambda^{-a}F\|_{L_t^2(I;L^2(X))}.
\end{Eq}

We apply these estimates to the first-order half-wave variables
\begin{Eq*}
\phi_+:=\partial_t\phi+iH\phi,\qquad
\phi_-:=\partial_t\phi-iH\phi.
\end{Eq*}
These are the standard diagonal variables for $\partial_t^2+H^2$.
They satisfy
\begin{Eq*}
(\partial_t-iH)\phi_+=F,\qquad
(\partial_t+iH)\phi_-=F,
\end{Eq*}
with initial data
\begin{Eq*}
\phi_+(0)=\psi_1+iH\psi_0,\qquad
\phi_-(0)=\psi_1-iH\psi_0.
\end{Eq*}
The same supersmoothing bound holds with $H$ replaced by $-H$, since
$\lambda^a(-H-z)^{-1}\lambda^a
=-\lambda^a(H+z)^{-1}\lambda^a$ and the uniform bound for $H$ over
$\C\setminus\R$ is invariant under $z\mapsto -z$.
The same retarded estimate is therefore valid for both half-wave groups
$e^{itH}$ and $e^{-itH}$.
By \eqref{Eq:appendix-hom-smoothing} and
\eqref{Eq:appendix-inhom-smoothing}, applied to $e^{itH}$ and $e^{-itH}$, we
obtain
\begin{Eq*}
\|\lambda^a\phi_\pm\|_{L_t^2(I;L^2(X))}
\lesssim
\|\psi_1\pm iH\psi_0\|_{L^2}
+
\|\lambda^{-a}F\|_{L_t^2(I;L^2(X))}.
\end{Eq*}
Since
\begin{Eq*}
\partial_t\phi=\frac{\phi_++\phi_-}{2},\qquad
H\phi=\frac{\phi_+-\phi_-}{2i},
\end{Eq*}
we conclude that
\begin{Eq*}
&\|\lambda^a\partial_t\phi\|_{L_t^2(I;L^2(X))}\\
&\quad+
\|\lambda^aH\phi\|_{L_t^2(I;L^2(X))}
\lesssim
\|H\psi_0\|_{L^2}
+
\|\psi_1\|_{L^2}
+
\|\lambda^{-a}F\|_{L_t^2(I;L^2(X))}.
\end{Eq*}
This proves \eqref{Eq:appendix-le}.
\end{proof}

\providecommand{\appendixproofpart}[1]{\par\medskip\noindent\textbf{#1.}}

\begin{proof}[Proof of \La{La:half-wave-res}]
All spectral measures and functional calculus below are those of the
self-adjoint operators $L$ and $H=L^{1/2}$.  The weights $\lambda^a$ are only
inserted as bounded multiplication operators on the left and right.  In
particular, no spectral calculus for a conjugated operator such as
$\lambda^aL\lambda^a$ or $\lambda^aH\lambda^a$ is involved.

\appendixproofpart{Part 1: Large imaginary part}
Fix the weight exponent $a>0$ from the statement.  Choose first a strip width
$\delta$ with $0<\delta<a$.  Applying \La{La:Wang-res} with $b=a$ and strip
width $\delta$ gives a high-frequency threshold $R$.  Next choose
$M>R$.  Applying \La{La:compact-LAP} on the low-frequency range
$0\leq \mu\leq M$ gives another strip width, say $\eta>0$.  We then set
\[
\varepsilon_0:=\min(\delta,\eta).
\]
With this common strip width, both the low-frequency and
high-frequency $L$-resolvent estimates are available.  
If
$|\operatorname{Im}z|\geq\varepsilon_0$, the spectral theorem gives
\[
\|(H-z)^{-1}\|_{L^2\to L^2}
\leq
|\operatorname{Im}z|^{-1}
\leq
\varepsilon_0^{-1}.
\]
Together with the $L^2$-boundedness of $\lambda^a$, this proves
\eqref{Eq:appendix-half-wave-res} outside the strip
$|\operatorname{Im}z|\leq\varepsilon_0$.  It remains to prove the uniform
estimate for
\[
z=\mu\pm i\varepsilon,\qquad \mu\in\R,\quad
0<\varepsilon\leq\varepsilon_0.
\]

\appendixproofpart{Part 2: The threshold line}
We first consider $\mu=0$.  For $s>0$, writing $\nu=s^{1/2}$, Stone's formula
gives
\begin{Eq*}
\lambda^a
\frac{\mathrm{d}E_L}{\mathrm{d}s}(s)\lambda^a
=
\lim_{\vep\to 0+}\frac{1}{2\pi i}\lambda^a
\left((L-(s+i\vep))^{-1}-(L-(s-i\vep))^{-1}\right)\lambda^a\ .
\end{Eq*}
in the weighted operator-valued sense.  
Notice that with $\vep\in (0, s/100)$, $\mu\in (\nu/2, 2\nu)$ and $\delta\in (0, \mu/10)$,
we could express $s\pm i\vep=(\mu\pm i\delta)^2$ so that
$(L-(s\pm i\vep))^{-1}=R(\mu\pm i\delta)$.
 Hence, by
\eqref{Eq:appendix-compact-lap} and the absence of point spectrum in the
present lemma, the weighted spectral measure of $L$ is absolutely continuous
on $(0,\nu_0^2]$ for $\nu_0=\min(M/2, 5\vep_0)$, and its sandwiched
operator-valued density satisfies
\begin{Eq*}
\left\|
\lambda^a\frac{\mathrm{d}E_L}{\mathrm{d}s}(s)\lambda^a
\right\|_{L^2\to L^2}
\leq C,
\qquad 0<s\leq \nu_0^2,
\end{Eq*}
with a uniform bound as $s\downarrow0$ by the threshold regularity contained in
\La{La:compact-LAP} and the no-pole assumption at the threshold.  Since
$H=L^{1/2}$, one has
$E_H([0,\nu])=E_L([0,\nu^2])$ and hence
\begin{Eq*}
\frac{\mathrm{d}E_H}{\mathrm{d}\nu}(\nu)
=
2\nu\,\frac{\mathrm{d}E_L}{\mathrm{d}s}(\nu^2).
\end{Eq*}
This identity is understood in the weighted operator-valued sense after
inserting the factors $\lambda^a$.  Therefore
\begin{Eq}\label{Eq:appendix-spectral-H-low}
\left\|
\lambda^a\frac{\mathrm{d}E_H}{\mathrm{d}\nu}(\nu)\lambda^a
\right\|_{L^2\to L^2}
\leq C\nu,
\qquad 0< \nu\leq \nu_0.
\end{Eq}
The factor $\nu$ is precisely what removes the possible half-wave singularity
at the threshold.  More explicitly, for $0<\nu<\nu_0$,
\begin{Eq*}
\lambda^a\,dE_H(\nu)\,\lambda^a
=
\lambda^a\frac{\mathrm dE_H}{\mathrm d\nu}(\nu)\lambda^a\,\mathrm d\nu
\end{Eq*}
in the weighted operator-valued sense, and the density has norm $O(\nu)$.
Since $\lambda^a$ is bounded and self-adjoint,
$\lambda^a\,dE_H(\nu)\,\lambda^a$ is a bounded positive operator-valued
measure.  The density estimate above controls its weighted total variation
near zero by $C\nu\,d\nu$.  Therefore, for
$0<\varepsilon\leq\varepsilon_0$,
\begin{Eq*}
\|\lambda^a(H-(0\pm i\vep))^{-1} \chi_{H\le\nu_0}\lambda^a\|_{L^2\to L^2}
=
\left\|
\int_0^{\nu_0}\frac{1}{\nu\mp i\varepsilon}
\lambda^a\,dE_H(\nu)\,\lambda^a
\right\|_{L^2\to L^2}
\lesssim
\int_0^{\nu_0}\frac{\nu}{|\nu\mp i\varepsilon|}\,d\nu
\lesssim1.
\end{Eq*}

On the other hand, it is trivial that $(H-(0\pm i\vep))^{-1} \chi_{H\ge\nu_0}$ is a bounded operator on $L^2(X)$, with norm bounded by $\nu_0^{-1}$.
Thus the functional calculus gives
\begin{Eq}\label{Eq:threshold-H-res}
\sup_{0<\varepsilon\leq\varepsilon_0}
\left\|
\lambda^a(H-(0\pm i\varepsilon))^{-1}\lambda^a
\right\|_{L^2\to L^2}
<\infty.
\end{Eq}

\appendixproofpart{Part 3: Negative real part}
We next prove the estimate in the region where the exterior half-wave
parameter has negative real part:
\begin{Eq}\label{Eq:H-negative-boundary}
\sup_{\mu<0,\ 0<\varepsilon\leq\varepsilon_0}
\left\|
\lambda^a(H-(\mu\pm i\varepsilon))^{-1}\lambda^a
\right\|_{L^2\to L^2}<\infty.
\end{Eq}
Indeed, by the spectral theorem,
\begin{Eq*}
\lambda^a(H-(\mu\pm i\varepsilon))^{-1}\lambda^a
=
\int_0^\infty(\nu-\mu\mp i\varepsilon)^{-1}
\lambda^a\,dE_H(\nu)\,\lambda^a.
\end{Eq*}
Near zero,
\begin{Eq*}
\int_0^{\nu_0}\frac{\nu}{|\nu-\mu\mp i\varepsilon|}\,d\nu
\leq
\int_0^{\nu_0}\frac{\nu}{\nu+|\mu|}\,d\nu
\lesssim1,
\end{Eq*}
while on $[\nu_0,\infty)$ the functional calculus gives a uniform bound because
$(\nu-\mu\mp i\varepsilon)^{-1}$ is bounded by $\nu_0^{-1}$ there.  This proves
\eqref{Eq:H-negative-boundary}.

\appendixproofpart{Part 4: Positive real part}
Let $z=\mu\pm i\varepsilon$ with $\mu>0$ and
$0<\varepsilon\leq\varepsilon_0$.  The algebraic identity
\begin{Eq*}
(H-z)^{-1}-(H+z)^{-1}=2z(L-z^2)^{-1}
\end{Eq*}
gives
\begin{Eq}\label{Eq:half-L-identity-boundary}
(H-(\mu\pm i\varepsilon))^{-1}
=
(H+\mu\pm i\varepsilon)^{-1}
+2(\mu\pm i\varepsilon)R(\mu\pm i\varepsilon).
\end{Eq}
Here, as in the preceding resolvent estimates, the notation
$R(\mu\pm i\varepsilon)$ means
$(L-(\mu\pm i\varepsilon)^2)^{-1}$.  After inserting $\lambda^a$ on both
sides, the first term is uniformly bounded by
\eqref{Eq:H-negative-boundary}, applied with exterior parameter $-\mu<0$ and
the opposite sign of $\varepsilon$.  We split the second term into bounded and
high frequencies.  If $0<\mu\leq M$, then
\eqref{Eq:appendix-compact-lap} gives
\begin{Eq*}
2|\mu\pm i\varepsilon|
\left\|
\lambda^aR(\mu\pm i\varepsilon)\lambda^a
\right\|_{L^2\to L^2}
\leq C_M.
\end{Eq*}
If $\mu\geq M$, then $\mu\geq R$, and hence \La{La:Wang-res} gives
\begin{Eq*}
2|\mu\pm i\varepsilon|
\left\|
\lambda^aR(\mu\pm i\varepsilon)\lambda^a
\right\|_{L^2\to L^2}
\lesssim
|\mu\pm i\varepsilon|\cdot|\mu\pm i\varepsilon|^{-1}
\lesssim1.
\end{Eq*}
Thus
\begin{Eq}\label{Eq:H-positive-boundary}
\sup_{\mu>0,\ 0<\varepsilon\leq\varepsilon_0}
\left\|
\lambda^a(H-(\mu\pm i\varepsilon))^{-1}\lambda^a
\right\|_{L^2\to L^2}<\infty.
\end{Eq}

\appendixproofpart{Part 5: Conclusion}
Combining Parts 2, 3, and 4, we obtain
\begin{Eq*}
\sup_{\mu\in\R,\ 0<\varepsilon\leq\varepsilon_0}
\left\|
\lambda^a(H-(\mu\pm i\varepsilon))^{-1}\lambda^a
\right\|_{L^2\to L^2}<\infty.
\end{Eq*}
Together with the large-imaginary-part estimate in Part 1, this proves
\eqref{Eq:appendix-half-wave-res}.
\end{proof}

\subsection*{Acknowledgment}
The authors would like to thank Chris Sogge and Yannick Sire for stimulating discussions regarding wave equations on negatively curved manifolds. They are also grateful to Dean Baskin, Xi Chen, Jared Wunsch, Andrew Hassell, and particularly to Fang Wang, for 
 their insights concerning resolvent estimates on conformally compact asymptotically hyperbolic manifolds.

The first author was supported in part by NSFC 12501292.
The second author was supported in part by NSFC 12141102.



\end{document}